\documentclass{amsart}%
\usepackage{graphicx}
\usepackage{amsmath}
\usepackage{amsfonts}
\usepackage{amssymb}
\usepackage{ifpdf}
\usepackage{esint}%
\providecommand{\U}[1]{\protect\rule{.1in}{.1in}}
\newtheorem{theorem}{Theorem}

\newtheorem{corollary}[theorem]{Corollary}

\newtheorem{definition}[theorem]{Definition}
\newtheorem{example}[theorem]{Example}

\newtheorem{proposition}[theorem]{Proposition}
\newtheorem{remark}[theorem]{Remark}

\makeatletter

\def\XXint#1#2#3{{\setbox0=\hbox{$#1{#2#3}{\int}$}
\vcenter{\hbox{$#2#3$}}\kern-.5\wd0}}

\makeatother
\makeatletter \@namedef{subjclassname@2020}{\textup{2020}
Mathematics Subject Classification} \makeatother
\begin{document}
\title[Oscillation inequalities and lower ball growth]{Oscillation inequalities, lower ball growth and Sobolev embeddings on metric
measure spaces}
\author{Joaquim Mart\'{\i}n}
\address{Department of Mathematics, Universitat Aut\`onoma de Barcelona,
08193 Bellaterra, Barcelona, Spain}
\email{Joaquin.Martin@uab.cat}
\thanks{Corresponding author: Joaquin.Martin@uab.cat. ORCID:
0000-0002-7467-787X}
\thanks{Funding: This work was partially supported by Grants
PID2024-160507NB-I00 and PID2024-155917NB-I00, funded by
MCIN/AEI/10.13039/501100011033.}

\begin{abstract}
We study the relation between lower ball growth, rearrangement oscillations,
and Sobolev-type embeddings on metric measure spaces. We prove that a lower
estimate for the ball function
\[
h_{\mu}(r):=\inf_{x\in\Omega}\mu(B(x,r))
\]
implies pointwise symmetrization inequalities for averaged differences and
Haj\l asz gradients, without doubling, Poincar\'e-type assumptions, or extra
topological hypotheses. Conversely, the validity of either family of
inequalities for suitable cutoff functions recovers the corresponding lower
growth estimate. Thus, up to multiplicative constants, the lower geometry of
the measure is equivalent to the pointwise oscillation principle itself.

Taking norms in rearrangement-invariant spaces yields normed oscillation
inequalities and Sobolev embeddings for Haj\l asz and averaged Besov spaces.
This separates the metric-measure input from the choice of the final
function-space target. In the power-growth model, the resulting targets are
identified with Lorentz spaces. We also show how Sobolev inequalities between
general rearrangement-invariant spaces force lower ball estimates through the
fundamental functions of the source and target spaces. The results apply to
general admissible growth functions and moduli of smoothness, beyond the
classical power setting.
\end{abstract}
\maketitle

\section{Introduction}

Let $(\Omega,d,\mu)$ be a metric measure space with $\mu(\Omega)=\infty$,
where $\mu$ is nonatomic and every ball has positive and finite measure. The
necessary notation and background are collected in
Section~\ref{sec:preliminaries}.

Sobolev-type embeddings on metric measure spaces are closely related to the
size of balls. A lower estimate of the form
\begin{equation}
\mu(B(x,r))\ge Cr^{Q}, \qquad x\in\Omega,\quad r>0, \label{eq:zero}%
\end{equation}
provides the dimensional information underlying the classical Sobolev
embedding theorem. Lower volume growth assumptions, often combined with
Poincar\'e-type conditions, are standard in the theory of Sobolev spaces on
metric measure spaces; see, for instance, \cite{Ha1,HK}.

There is also a converse direction. P. G\'orka \cite{GorkaPA} proved that,
under a doubling assumption, an embedding
\begin{equation}
M^{1,p}(\Omega)\hookrightarrow L^{q}(\Omega), \qquad p<q, \label{incluintro}%
\end{equation}
implies \eqref{eq:zero}, with
\[
\frac{1}{Q}=\frac{1}{p}-\frac{1}{q}.
\]
Here $M^{1,p}(\Omega)$ denotes the Haj\l asz--Sobolev space based on $L^{p}$;
see Section~\ref{sec:oscillation-reductions}. Related converse results have
subsequently been obtained for Haj\l asz--Sobolev, Besov, and Slobodeckij-type
spaces; see \cite{AGH,GSStudia,GorkaNA,GKS,MO} and the references therein.

The purpose of this paper is to identify and develop the oscillation mechanism
underlying these results. Rather than starting from a prescribed target space,
such as the $L^{q}(\Omega)$ target in \eqref{incluintro}, we first establish
pointwise symmetrization inequalities and only afterwards identify the
resulting function spaces. This separates the geometric part of the argument
from the choice of the final target.

The central quantity is the rearrangement oscillation
\[
O(f,t):=f^{\ast\ast}(t)-f^{\ast}(t),
\]
where $f^{\ast\ast}(t)=\frac{1}{t}\int_{0}^{t}f^{\ast}(s)\,ds$ and $f^{\ast}$
is the decreasing rearrangement of $f$ with respect to the measure $\mu.$

The systematic use of $f^{\ast\ast}-f^{\ast}$ in symmetrization and
endpoint Sobolev inequalities goes back to work of M. Milman and
collaborators; see, for example, \cite{MP,MMP,MM6,MM3,BMR,MMas,MOfe}
and the references therein.

Given an admissible growth function $v$ in the sense of
Subsection~\ref{subsec:fundamental-indices}, we study the consequences of the
lower estimate\footnote{The standard model is $v(r)=cr^{Q}$, corresponding to
a lower Ahlfors-type condition. Such assumptions appear in many Sobolev
embedding results on metric measure spaces; see, for instance,
\cite{GorkaPA,AGH,Karak,GorkaNA,GSStudia,PA}. General moduli of smoothness in
metric spaces were considered, for example, in \cite{Mas,GKS}.}%

\[
h_{\mu}(r)\ge v(r),\qquad r>0.
\]
The main symmetrization result shows that this lower bound implies oscillation
inequalities for two different notions of smoothness.

First, let $0<p<\infty$ and set $p_{0}:=\min\{p,1\}$. For the averaged moduli
of smoothness, we use
\[
\nabla_{r}^{p} f(x) := \left(  \frac{1}{\mu(B(x,r))} \int_{B(x,r)}
|f(x)-f(y)|^{p}\,d\mu(y) \right)  ^{1/p}, \qquad r>0.
\]
The lower estimate $h_{\mu}\ge v$ yields, in a simplified form,
\begin{equation}
O(|f|^{p_{0}},t)^{1/p_{0}} \leq C \left(  \left[  \nabla_{v^{-1}(2t)}^{p}
f\right]  ^{p} \right)  ^{**}(t)^{1/p}. \label{osc1intro}%
\end{equation}

Second, for the Haj\l asz scale, let $\phi$ be an admissible
modulus.\footnote{The classical Haj\l asz--Sobolev space corresponds to
$\phi(t)=t$; see \cite{Ha1,HK}. The fractional model corresponds to
$\phi(t)=t^{s}$, $s>0$, and was further developed in metric settings in
\cite{Yang}. General moduli of smoothness and related generalized smoothness
scales appear, for example, in \cite{Mas,Karak,LYY}.} A function $g$ is a
$\phi$-Haj\l asz gradient of $f$ if
\[
|f(x)-f(y)| \le\phi(d(x,y))\bigl(g(x)+g(y)\bigr)
\]
for almost every $x,y\in\Omega$. Setting
\[
\omega(t):=\phi(v^{-1}(t)),
\]
we obtain
\begin{equation}
O(|f|^{p_{0}},t)^{1/p_{0}} \leq C \omega(t)(g^{p})^{**}(t)^{1/p}.
\label{osc2intro}%
\end{equation}

The important point is that both \eqref{osc1intro} and \eqref{osc2intro}
require only the lower ball estimate $h_{\mu}\ge v$. In particular,
\eqref{osc2intro} extends the result obtained in \cite{PA} for the model power
setting $\phi(r)=r^{\alpha}$, with lower volume growth of order $r^{Q}$ and
under additional structural assumptions on the measure (doubling or suitable
continuity assumptions).

Conversely, the validity of either \eqref{osc1intro} or \eqref{osc2intro} for
measurable functions implies\footnote{Throughout the paper, $A\preceq B$ means
that $A\leq C B$ for a constant $C>0$ independent of the relevant functions
and variables. We write $A\succeq B$ when $B\preceq A$, and $A\simeq B$ when
both $A\preceq B$ and $A\succeq B$.}
\[
h_{\mu}(r)\succeq v(r).
\]

Thus, up to multiplicative constants, the lower geometry and the pointwise
oscillation inequalities contain the same information.

Let $X$ be a rearrangement-invariant space. Taking the $X$-norm in
\eqref{osc2intro}, we obtain the Sobolev-type oscillation estimate
\[
\left\Vert \frac{O(|f|^{p_{0}},\cdot)^{1/p_{0}}}{\omega(\cdot)}\right\Vert
_{X}\preceq\left\Vert \left(  (g^{p})^{\ast\ast}(\cdot)\right)  ^{1/p}%
\right\Vert _{X}.
\]
As shown in Section~\ref{sec:oscillation-reductions}, the auxiliary exponent
$p$ can be chosen so that the right-hand side is controlled by $\Vert
g\Vert_{X}$. Hence
\[
\left\Vert \frac{O(|f|^{p_{0}},\cdot)^{1/p_{0}}}{\omega(\cdot)}\right\Vert
_{X}\preceq\Vert g\Vert_{X}.
\]

Conversely, testing the resulting Sobolev-type oscillation inequalities on
suitable cutoff functions recovers the corresponding lower estimate for
$h_{\mu}$. Thus, the results follow the general scheme
\[
\text{lower ball growth} \quad\Longleftrightarrow\quad\text{pointwise
oscillation} \quad\Longleftrightarrow\quad\text{normed oscillation}.
\]

To see how $\omega$ determines the correct target space, consider the basic
power model
\[
v(r)\simeq r^{Q},\qquad\phi(r)=r^{\alpha}.
\]
Then
\[
\omega(t)=\phi(v^{-1}(t))\simeq t^{\alpha/Q}.
\]

Suppose now that
\[
\underline{\alpha}_{X}>\frac{\alpha}{Q},
\]
where \(\underline{\alpha}_X\) denotes the lower Boyd index of \(X\)
(see \eqref{boydindi} below). Then, as shown in
Section~\ref{sec:embedding-consequences}, \[
\left\|t^{-\alpha/Q}O(|f|^{p_0},t)^{1/p_0}\right\|_X \simeq
\left\|t^{-\alpha/Q}f^{**}(t)\right\|_X. \]
 Consequently, the normed
oscillation inequality is equivalent to the Sobolev-type estimate
\[
\left\|  t^{-\alpha/Q}f^{**}(t)\right\|  _{X} \preceq\|g\|_{X}.
\]
Thus, by the equivalence discussed above, this Sobolev-type estimate is
equivalent to the lower growth condition
\[
h_{\mu}(r)\succeq r^{Q}.
\]

This allows us to recover, as a particular case, the equivalences between
lower ball bounds and Sobolev embeddings obtained in
\cite{AGH,GSStudia,GorkaNA,GKS,MO}. Indeed, let $X=L^{p}$, with $1<p<Q/\alpha
$. Since
\[
\underline{\alpha}_{L^{p}} = \overline{\alpha}_{L^{p}} = \frac1p,
\]
the condition $\underline{\alpha}_{X}>\alpha/Q$ is equivalent to $\alpha p<Q$.
Defining $p_{\alpha}^{*}$ by
\[
\frac1{p_{\alpha}^{*}} = \frac1p-\frac{\alpha}{Q},
\]
we have
\[
\left\|  t^{-\alpha/Q}f^{**}(t)\right\|  _{L^{p}(0,\infty)} \simeq
\|f\|_{L^{p_{\alpha}^{*},p}(\Omega)}.
\]
Hence the normed oscillation estimate becomes the sharper Lorentz embedding
\[
\|f\|_{L^{p_{\alpha}^{*},p}(\Omega)} \preceq\|g\|_{L^{p}(\Omega)}.
\]

The same scheme applies to the averaged Besov scale. Starting from
\eqref{osc1intro}, we obtain a corresponding normed oscillation estimate
controlled by the averaged Besov seminorm, while the converse implication
recovers the lower ball growth. The only additional assumption is the uniform
boundedness of the averaging operators on the appropriate
rearrangement-invariant space; see Section~\ref{sec:oscillation-reductions}.

Taken together, these results show that the lower geometry of the space is
already encoded in the rearrangement oscillation, before any target space is
selected. The pointwise inequalities provide a common mechanism for the
Haj\l asz and averaged Besov scales, while the subsequent choice of a
rearrangement-invariant target determines how much of this geometric
information remains visible at the level of embeddings.

Let us finally describe the organization of the paper.
Section~\ref{sec:preliminaries} recalls the necessary background on
rearrangements, rearrangement-invariant spaces, admissible
functions, Boyd indices, Haj\l asz gradients, and averaged moduli of
smoothness. Section~\ref{sec:symmetrization} proves the pointwise
symmetrization inequalities for averaged differences and Haj\l asz
gradients, together with their converse implications.
Section~\ref{sec:oscillation-reductions} introduces the Haj\l asz
and averaged Besov scales, derives the corresponding normed
oscillation inequalities, and proves that they retain the lower ball
growth. Section~\ref{sec:embedding-consequences} identifies the
Lorentz targets, studies the lower growth forced by Sobolev
inequalities between general rearrangement-invariant spaces through
the fundamental functions of the source and target spaces, compares
the resulting conclusions with \cite{AGH}, and derives
Slobodeckij-type consequences. The final section compares the Haj\l
asz and averaged Besov scales and gives non-doubling examples for
which the averaging operators remain uniformly bounded.

\section{Preliminaries and background}

\label{sec:preliminaries}

In this section we fix some notation and recall the background material that
will be used throughout the paper. We also specify the basic assumptions on
the metric measure spaces and on the rearrangement-invariant spaces considered
below. For further background on rearrangements, rearrangement-invariant
spaces, interpolation, Boyd indices, we refer to \cite{BS,Boy,MS,KPS}.

Throughout, $(\Omega,d)$ denotes a separable metric space. For $x\in\Omega$
and $r>0$, we write $B(x,r):=\{y\in\Omega:\ d(x,y)<r\} $ for the open ball of
center $x$ and radius $r$.

A \emph{metric measure space} is a triple $(\Omega,d,\mu)$ with $\mu$ a Borel
measure such that $0<\mu(B)<\infty$ for every ball $B\subset\Omega$. We assume
that $\mu(\Omega)=\infty$ and that $\mu(\left\{  x\right\}  )=0 $ for all
$x\in\Omega$, a convention we keep for the rest of the paper.

Given a metric measure space $(\Omega,d,\mu)$, define the \emph{lower ball
function} $h_{\mu}:(0,\infty)\to[0,\infty)$ by
\[
h_{\mu}(r):=\inf_{x\in\Omega}\mu(B(x,r)).
\]

We shall use the standard notation
\[
\mathchoice
{{\setbox0=\hbox{$\displaystyle{\textstyle -}{\int}$}
\vcenter{\hbox{$\textstyle -$}}\kern-.5\wd0}}
{{\setbox0=\hbox{$\textstyle{\scriptstyle -}{\int}$}
\vcenter{\hbox{$\scriptstyle -$}}\kern-.5\wd0}}
{{\setbox0=\hbox{$\scriptstyle{\scriptscriptstyle -}{\int}$}
\vcenter{\hbox{$\scriptscriptstyle -$}}\kern-.5\wd0}}
{{\setbox0=\hbox{$\scriptscriptstyle{\scriptscriptstyle -}{\int}$}
\vcenter{\hbox{$\scriptscriptstyle -$}}\kern-.5\wd0}} \!\int_{E} f\,d\mu:=
\frac{1}{\mu(E)}\int_{E} f\,d\mu,
\]
whenever $0<\mu(E)<\infty$.

\subsection{Rearrangements and oscillations}

\label{subsec:rearrangements}

Let $(\Omega,d,\mu)$ be a metric measure space, and let $L^{0}(\mu)$ denote
the space of measurable functions modulo equality $\mu$-a.e.

For $f\in L^{0}(\mu)$, the decreasing rearrangement of $f$ is defined by
\[
f_{\mu}^{\ast}(t):=\inf\{\lambda\geq0:\ \mu\{x\in\Omega:\ |f(x)|>\lambda\}\leq
t\},\qquad t>0.
\]
When the underlying measure is clear from the context, we shall simply write
$f^{\ast}$.

A basic property of rearrangements is
\begin{equation}
\label{hard}\sup_{\mu(E)=t}\int_{E} |f(x)|\,d\mu(x) = \int_{0}^{t}
f^{*}(s)\,ds.
\end{equation}
The maximal rearrangement of $f$ is defined by
\[
f^{**}(t) := \frac1t\int_{0}^{t} f^{*}(s)\,ds, \qquad t>0.
\]
Then $f^{**}$ is decreasing and
\[
f^{*}(t)\le f^{**}(t), \qquad t>0.
\]

The oscillation functional is defined by
\[
O(f,t):=f^{\ast\ast}(t)-f^{\ast}(t),\qquad t>0.
\]
Then one has
\begin{equation}
\frac{d}{dt}f^{\ast\ast}(t)=-\frac{O(f,t)}{t}. \label{der2est}%
\end{equation}
Moreover, the function $t\mapsto tO(f,t)$ is increasing. Consequently, for
every $u>0$,
\begin{align*}
f^{\ast\ast}(u)-f^{\ast\ast}(2u)  &  =\int_{u}^{2u}O(f,s)\,\frac{ds}{s}%
=\int_{u}^{2u}sO(f,s)\,\frac{ds}{s^{2}}\\
&  \geq uO(f,u)\int_{u}^{2u}\frac{ds}{s^{2}}=\frac{1}{2}\,O(f,u) .
\end{align*}
Equivalently,
\begin{equation}
O(f,u)\leq2\left(  f ^{\ast\ast}(u)-f^{\ast\ast}(2u)\right)  .
\label{eq:osc-drop}%
\end{equation}

\subsection{Fundamental indices and admissible functions}

\label{subsec:fundamental-indices}

We denote by $\mathcal{B}$ the class of continuous functions
\[
\phi:[0,\infty)\to[0,\infty)
\]
such that $\phi(t)>0$ for every $t>0$, and
\[
M_{\phi}(\lambda) := \sup_{s>0}\frac{\phi(\lambda s)}{\phi(s)}%
\]
is finite for every $\lambda>0$.

Given $\phi\in\mathcal{B}$, its lower and upper fundamental indices are
defined by
\begin{equation}
\underline{\beta}_{\phi}= \lim_{\lambda\to0^{+}} \frac{\log M_{\phi}(\lambda
)}{\log\lambda}, \qquad\overline{\beta}_{\phi}= \lim_{\lambda\to\infty}
\frac{\log M_{\phi}(\lambda)}{\log\lambda}. \label{indi}%
\end{equation}

We denote by $\mathcal{A}$ the subclass of $\mathcal{B}$ consisting of
increasing functions satisfying
\[
\phi(0)=0.
\]
If $\phi\in\mathcal{A}$, then
\[
0\leq\underline{\beta}_{\phi}\leq\overline{\beta}_{\phi}\leq\infty.
\]
Finally, $\mathcal{A}_{0}$ denotes the subclass of $\mathcal{A}$ formed by
those functions for which
\[
0<\underline{\beta}_{\phi}\leq\overline{\beta}_{\phi}<\infty.
\]

\begin{remark}
\label{rem0} Given $\phi\in\mathcal{A}_{0}$, we shall use the standard
equivalence
\begin{equation}
\phi(t)\simeq\int_{0}^{t} \phi(s)\,\frac{ds}{s} . \label{eqAo}%
\end{equation}
Since functions in $\mathcal{A}_{0}$ are considered up to equivalence, we
shall always choose convenient representatives. In particular, whenever
inverse functions are used, we choose strictly increasing representatives of
$v$ and $\phi$. Their inverses then belong again to $\mathcal{A}_{0}$.

Moreover, replacing a function $\phi\in\mathcal{A}_{0}$ by the equivalent
representative
\[
\widetilde{\phi}(t):=\int_{0}^{t} \phi(s)\,\frac{ds}{s},
\]
we may also assume, whenever needed, that $\phi$ is differentiable and
satisfies
\[
t\phi^{\prime}(t)\simeq\phi(t),\qquad t>0.
\]
The same convention will be used for $v$. This convention does not affect the
results obtained in the following sections, except for harmless multiplicative constants.
\end{remark}

\subsection{Rearrangement-invariant spaces}

\label{subsec:ri-spaces}

Let $(\Omega,d,\mu)$ be a metric measure space. A Banach linear subspace
$X_{\mu}(\Omega)\subset L^{0}(\mu)$ is called a \emph{Banach function space}
if it satisfies:

\begin{enumerate}
\item[(i)] (\textbf{Lattice property}) If $g\in X_{\mu}(\Omega)$ and $f\in
L^{0}(\mu)$ with $|f|\le|g|$, then $f\in X_{\mu}(\Omega)$ and $\|f\|_{ X_{\mu
}(\Omega)}\le\|g\|_{X_{\mu}(\Omega)}$.

\item[(ii)] (\textbf{Fatou property}) If $0\leq f_{n}\uparrow f$ almost
everywhere, then $\Vert f_{n}\Vert_{X_{\mu}(\Omega)}\uparrow\Vert
f\Vert_{X_{\mu}(\Omega)}$.
\end{enumerate}

A Banach function space $X_{\mu}(\Omega)$ is called
\emph{rearrangement-invariant} (r.i. space in short) if $f^{\ast}=g^{\ast}$
and $f\in X_{\mu}(\Omega)$ imply $g\in X_{\mu}(\Omega)$ and
\[
\Vert g\Vert_{X_{\mu}(\Omega)}=\Vert f\Vert_{X_{\mu}(\Omega)}.
\]
Typical examples are Lebesgue, Lorentz, Marcinkiewicz, Lorentz--Zygmund,
Orlicz and Lorentz--Orlicz spaces.

To simplify notation, we shall usually write $X$ instead of $X_{\mu}(\Omega)$
when the underlying measure space is clear.

For every r.i. space $X$,
\[
L^{\infty}\cap L^{1} \subset X \subset L^{1}+L^{\infty}%
\]
with continuous embeddings.

Every r.i. space $X$ admits a representation space $\overline X=\overline
X(0,\infty)$, defined with respect to Lebesgue measure, such that
\[
\|f\|_{X}=\|f^{*}\|_{\overline X}.
\]
For this standard representation, see \cite[Theorem 4.10 and the subsequent
remarks]{BS}. To simplify notation, we shall write $\|f^{*}\|_{X}$ instead of
$\|f^{*}\|_{\overline X}$ whenever no confusion can arise.

We shall also use the following monotonicity property. If $f,g\in L^{0}(\mu)$
satisfy
\[
\int_{0}^{r} f^{*}(s)\,ds \le\int_{0}^{r} g^{*}(s)\,ds, \qquad r>0,
\]
then
\[
\|f\|_{X}\le\|g\|_{X}%
\]
for every r.i. space $X$.

The associate space $X^{\prime}$ of $X$ is the space of all measurable
functions $h$ such that
\[
\int_{\Omega}|h(x)g(x)|\,d\mu(x)<\infty\qquad\text{for every }g\in X,
\]
endowed with the norm
\[
\|h\|_{X^{\prime}} := \sup_{\|g\|_{X}\le1} \int_{\Omega}|h(x)g(x)|\,d\mu
(x)=\sup_{\|g\|_{X}\le1} \int_{0}^{\infty}h^{*}(s)g^{*}(s)\,ds.
\]
Then $X^{\prime}$ is again an r.i. space, and the H\"{o}lder inequality
\begin{equation}
\label{Hold}\int_{\Omega}|f(x)g(x)|\,d\mu(x) \le\|f\|_{X}\|g\|_{X^{\prime}}%
\end{equation}
holds for every $f\in X$ and $g\in X^{\prime}$.

The fundamental function of an r.i. space $X$ is defined by
\begin{equation}
\varphi_{X}(t) := \|\chi_{E}\|_{X}, \qquad\mu(E)=t. \label{fundamentalf}%
\end{equation}
The definition is independent of the set $E$, by rearrangement invariance. The
function $\varphi_{X}$ is quasi-concave and satisfies the duality relation
\begin{equation}
\label{si}\varphi_{X}(t)\varphi_{X^{\prime}}(t)=t.
\end{equation}

Let $0<p<\infty$, and let $X$ be an r.i. space on $\Omega$. The $p$%
-convexification $X^{(p)}$ is defined by
\[
X^{(p)}:=\{f\in L^{0}(\mu): |f|^{p}\in X\},
\]
with
\[
\|f\|_{X^{(p)}}:=\||f|^{p}\|_{X}^{1/p}.
\]
For simplicity, $X^{(1)}=X$. Notice that $X^{(p)}$ is again a
rearrangement-invariant function space; it may be only quasi-Banach when
$0<p<1$. For example,
\[
(L^{1})^{(p)}=L^{p}.
\]

Classically, conditions on r.i. spaces are given in terms of the Hardy
operators
\[
P^{(q)}f(t) = \left(  \frac1t\int_{0}^{t} |f(x)|^{q}\,dx \right)  ^{1/q},
\qquad Q_{\lambda}^{(q)}f(t) = \left(  \frac1{t^{\lambda}} \int_{t}^{\infty
}|f(x)|^{q}\,\frac{dx}{x^{1-\lambda}} \right)  ^{1/q},
\]
where $0<q<\infty$ and $0\le\lambda<1$.

The boundedness of these operators on r.i. spaces can be described in terms of
the Boyd indices. Let
\[
h_{X}(s):=\Vert E_{s}\Vert_{X\rightarrow X},\qquad s>0,
\]
where $E_{s}$ is the dilation operator on the representation space
$\overline{X}$, defined by
\[
E_{s}f(t)=f\left(  \frac{t}{s}\right)  ,
\]
The Boyd indices of $X$ are the lower and upper fundamental indices of the
dilation function $h_{X}$. Thus
\begin{equation}
\overline{\alpha}_{X}=\inf_{s>1}\frac{\log h_{X}(s)}{\log s},\qquad
\underline{\alpha}_{X}=\sup_{0<s<1}\frac{\log h_{X}(s)}{\log s}.
\label{boydindi}%
\end{equation}
For example, if $X=L^{p}$, $1<p<\infty$, then
\[
\overline{\alpha}_{X}=\underline{\alpha}_{X}=\frac{1}{p}.
\]

It is well known that, see \cite{MMjfsa,MS},
\begin{equation}
\overline{\alpha}_{X}<\frac1q \quad\Longleftrightarrow\quad P^{(q)}\text{ is
bounded on }X, \label{Boydsup}%
\end{equation}
and
\begin{equation}
\underline{\alpha}_{X}>\frac{\lambda}{q} \quad\Longleftrightarrow\quad
Q_{\lambda}^{(q)}\text{ is bounded on }X. \label{Boydinf}%
\end{equation}

\subsection{Haj\l asz gradients and averaged moduli of smoothness}

\label{subsec:hajlasz-gradients-moduli}

Let $\phi\in\mathcal{A}_{0}$ and let $f\in L^{0}(\mu)$. A non-negative
measurable function $g$ is called a \emph{$\phi$--Haj\l asz gradient} of $f$
if
\[
|f(x)-f(y)|\leq\phi(d(x,y))\left(  g(x)+g(y)\right)  \qquad\text{for $\mu
$-a.e. }x,y\in\Omega.
\]
We denote by $D^{\phi}(f)$ the collection of all $\phi$--Haj\l asz gradients
of $f$.

Let $0<p<\infty$. Given $f\in L_{\mathrm{loc}}^{p}(\mu)$ and $r>0$, we define
the averaged moduli of smoothness
\[
\nabla_{r}^{p}f(x):=\left(
\mathchoice {{\setbox0=\hbox{$\displaystyle{\textstyle -}{\int}$} \vcenter{\hbox{$\textstyle -$}}\kern-.5\wd0}} {{\setbox0=\hbox{$\textstyle{\scriptstyle -}{\int}$} \vcenter{\hbox{$\scriptstyle -$}}\kern-.5\wd0}} {{\setbox0=\hbox{$\scriptstyle{\scriptscriptstyle -}{\int}$} \vcenter{\hbox{$\scriptscriptstyle -$}}\kern-.5\wd0}} {{\setbox0=\hbox{$\scriptscriptstyle{\scriptscriptstyle -}{\int}$} \vcenter{\hbox{$\scriptscriptstyle -$}}\kern-.5\wd0}}
\!\int_{B(x,r)}|f(x)-f(y)|^{p}\,d\mu(y)\right)  ^{1/p}.
\]
If $X$ is an r.i. space on $\Omega$, the $X$-modulus of smoothness associated
with $\nabla_{r}^{p}$ is defined by
\[
E_{X}^{p}(f,r):=\Vert\nabla_{r}^{p}f\Vert_{X}.
\]

\begin{remark}
\label{rem1} In the Euclidean setting these definitions recover the usual
Sobolev and Besov objects.

First, Haj\l asz gradients contain the classical Sobolev gradient through the
Hardy--Littlewood maximal operator. More precisely, by the standard pointwise
Sobolev estimate of Haj\l asz \cite{Ha1} and Haj\l asz and Koskela \cite{HK},
if $f\in W_{\mathrm{loc}}^{1,1}(\mathbb{R}^{n})$, then
\[
g(x):=C\,M(|\nabla f|)(x)
\]
is a $\phi$--Haj\l asz gradient of $f$ with $\phi(t)=t$. That is,
\[
|f(x)-f(y)|\leq C\,|x-y|\left(  M(|\nabla f|)(x)+M(|\nabla f|)(y)\right)
\]
for a.e. $x,y\in\mathbb{R}^{n}$.

Second, in the Euclidean setting, let $X=L^{p}(\mathbb{R}^{n})$ and write
\[
E_{p}(f,t):=E_{L^{p}}^{p}(f,t)=\|\nabla_{t}^{p} f\|_{L^{p}(\mathbb{R}^{n})}.
\]
Then
\[
E_{p}(f,t) = \left(  \int_{\mathbb{R}^{n}}
\mathchoice {{\setbox0=\hbox{$\displaystyle{\textstyle -}{\int}$} \vcenter{\hbox{$\textstyle -$}}\kern-.5\wd0}} {{\setbox0=\hbox{$\textstyle{\scriptstyle -}{\int}$} \vcenter{\hbox{$\scriptstyle -$}}\kern-.5\wd0}} {{\setbox0=\hbox{$\scriptstyle{\scriptscriptstyle -}{\int}$} \vcenter{\hbox{$\scriptscriptstyle -$}}\kern-.5\wd0}} {{\setbox0=\hbox{$\scriptscriptstyle{\scriptscriptstyle -}{\int}$} \vcenter{\hbox{$\scriptscriptstyle -$}}\kern-.5\wd0}}
\!\int_{B(x,t)} |f(x)-f(y)|^{p}\,dy\,dx \right)  ^{1/p}.
\]
Since $|B(x,t)|=|B(0,t)|$, the change of variables $y=x+h$ gives
\[
E_{p}(f,t) = \left(
\mathchoice {{\setbox0=\hbox{$\displaystyle{\textstyle -}{\int}$} \vcenter{\hbox{$\textstyle -$}}\kern-.5\wd0}} {{\setbox0=\hbox{$\textstyle{\scriptstyle -}{\int}$} \vcenter{\hbox{$\scriptstyle -$}}\kern-.5\wd0}} {{\setbox0=\hbox{$\scriptstyle{\scriptscriptstyle -}{\int}$} \vcenter{\hbox{$\scriptscriptstyle -$}}\kern-.5\wd0}} {{\setbox0=\hbox{$\scriptscriptstyle{\scriptscriptstyle -}{\int}$} \vcenter{\hbox{$\scriptscriptstyle -$}}\kern-.5\wd0}}
\!\int_{B(0,t)} \|f(\cdot+h)-f(\cdot)\|_{L^{p}(\mathbb{R}^{n})}^{p} \,dh
\right)  ^{1/p}.
\]
Consequently,
\[
E_{p}(f,t) \le\sup_{|h|\le t} \|f(\cdot+h)-f(\cdot)\|_{L^{p}(\mathbb{R}^{n})}
=:\omega_{p}(f,t).
\]
Conversely, the corresponding Besov seminorms defined through $E_{p}(f,t)$ are
equivalent to the classical ones by the standard averaging argument; see, for
instance \cite{GKS}.
\end{remark}

\section{Symmetrization inequalities and lower bounds for the measure}

\label{sec:symmetrization}

Throughout this section, $v\in\mathcal{A}_{0}$ denotes an admissible lower
growth function and $\phi\in\mathcal{A}_{0}$ denotes the Haj\l asz modulus. We
assume that
\begin{equation}
\label{eq:lower-ball-v}h_{\mu}(r):=\inf_{x\in\Omega}\mu(B(x,r))\ge v(r),
\qquad r>0.
\end{equation}

Given a locally integrable function $h$, we denote by $A_{r}h$ the averaging
operator defined by
\[
A_{r}%
h(x):=\mathchoice {{\setbox0=\hbox{$\displaystyle{\textstyle -}{\int}$} \vcenter{\hbox{$\textstyle -$}}\kern-.5\wd0}} {{\setbox0=\hbox{$\textstyle{\scriptstyle -}{\int}$} \vcenter{\hbox{$\scriptstyle -$}}\kern-.5\wd0}} {{\setbox0=\hbox{$\scriptstyle{\scriptscriptstyle -}{\int}$} \vcenter{\hbox{$\scriptscriptstyle -$}}\kern-.5\wd0}} {{\setbox0=\hbox{$\scriptscriptstyle{\scriptscriptstyle -}{\int}$} \vcenter{\hbox{$\scriptscriptstyle -$}}\kern-.5\wd0}}
\!\int_{B(x,r)}h(y)\,d\mu(y),\qquad r>0.
\]

\begin{theorem}
\label{thm:symmetrization} Let $0<p<\infty$, and set
\[
p_{0}:=\min\{1,p\}.
\]
Then, for every $f\in L^{p_{0}}(\mu)+L^{\infty}(\mu)$, and every $t>0$, the
following estimates hold.

\begin{enumerate}
\item
\begin{equation}
\label{eq:symmetrization-gradient}O(|f|^{p_{0}},t)^{1/p_{0}} \preceq\left(
\left[  \nabla_{v^{-1}(2t)}^{p}f \right]  ^{p} \right)  ^{**}(t)^{1/p}.
\end{equation}

\item Let $\phi\in\mathcal{A}_{0}$, and set
\[
\omega(t):=\phi(v^{-1}(t)),\qquad t>0.
\]
If $g\in D^{\phi}(f)$, then
\begin{equation}
\label{eq:symmetrization-hajlasz-root}O(|f|^{p_{0}},t)^{1/p_{0}} \preceq
\omega(t)\,(g^{p})^{**}(t)^{1/p}.
\end{equation}
Equivalently,
\begin{equation}
\label{eq:symmetrization-hajlasz}O(|f|^{p_{0}},t) \preceq\omega(t)^{p_{0}%
}\,(g^{p})^{**}(t)^{p_{0}/p}.
\end{equation}

\end{enumerate}
\end{theorem}

\begin{proof}
We prove the two estimates separately.

\medskip

\noindent\textit{Proof of (i).} Fix $t>0$ and set
\[
r:=v^{-1}(2t).
\]
By (\ref{eq:lower-ball-v}),
\[
\mu(B(x,r))\geq v(r)=2t,\qquad x\in\Omega.
\]

We distinguish two cases.

Assume first that $0<p\le1$. Then $p_{0}=p$. Since
\[
|f(x)|^{p}\leq|f(x)-f(y)|^{p}+|f(y)|^{p},
\]
averaging this inequality over $B(x,r)$, we obtain
\[
|f(x)|^{p} \leq[\nabla_{r}^{p}f(x)]^{p}+A_{r}(|f|^{p})(x).
\]
By (\ref{hard}), using that $2t\leq\mu(B(x,r))$ and that $(|f|^{p})^{**}$ is
decreasing, we get
\begin{align}
A_{r}(|f|^{p})(x)  &  \leq(|f|^{p})^{**}(\mu(B(x,r)))\label{eq:media}\\
&  \leq(|f|^{p})^{**}(2t).\nonumber
\end{align}
Hence
\begin{equation}
\label{eq:sym-proof-first-bound}|f(x)|^{p} \leq[\nabla_{r}^{p}f(x)]^{p} +
(|f|^{p})^{**}(2t).
\end{equation}

Let $E\subset\Omega$ be measurable with $\mu(E)=t$. Integrating
(\ref{eq:sym-proof-first-bound}) over $E$, we get
\[
\int_{E}|f(x)|^{p}\,d\mu(x) \leq\int_{E}[\nabla_{r}^{p}f(x)]^{p}\,d\mu(x) +
t(|f|^{p})^{**}(2t).
\]
Taking the supremum over all such $E$, and using (\ref{hard}), yields
\[
t(|f|^{p})^{**}(t) \leq t\left(  [\nabla_{r}^{p}f]^{p}\right)  ^{**}(t) +
t(|f|^{p})^{**}(2t).
\]
Therefore
\[
(|f|^{p})^{**}(t)-(|f|^{p})^{**}(2t) \leq\left(  [\nabla_{r}^{p}f]^{p}\right)
^{**}(t).
\]
Using (\ref{eq:osc-drop}), we conclude that
\[
O(|f|^{p},t) \preceq\left(  [\nabla_{r}^{p}f]^{p}\right)  ^{**}(t).
\]
Since $p_{0}=p$ and $r=v^{-1}(2t)$, this gives
\[
O(|f|^{p_{0}},t)^{1/p_{0}} \preceq\left(  \left[  \nabla_{v^{-1}(2t)}^{p}f
\right]  ^{p} \right)  ^{**}(t)^{1/p}.
\]

Assume now that $p>1$. Then $p_{0}=1$. Averaging the triangle inequality over
$B(x,r)$, and using H\"older's inequality, gives
\[
|f(x)| \leq\nabla_{r}^{p}f(x) + A_{r}(|f|)(x).
\]
Again, by (\ref{hard}), using that $2t\leq\mu(B(x,r))$ and that $f^{**}$ is
decreasing,
\[
A_{r}(|f|)(x) \leq f^{**}(\mu(B(x,r))) \leq f^{**}(2t).
\]
Hence
\begin{equation}
\label{eq:sym-proof-first-bound-pgreater1}|f(x)| \leq\nabla_{r}^{p}%
f(x)+f^{**}(2t).
\end{equation}
Let $E\subset\Omega$ be measurable with $\mu(E)=t$. Integrating
(\ref{eq:sym-proof-first-bound-pgreater1}) over $E$, taking the supremum over
all such $E$, and using (\ref{hard}), we obtain
\[
t f^{**}(t) \leq t(\nabla_{r}^{p}f)^{**}(t) + t f^{**}(2t).
\]
Thus
\[
f^{**}(t)-f^{**}(2t) \leq(\nabla_{r}^{p}f)^{**}(t).
\]
By H\"older's inequality in rearranged form,
\[
(\nabla_{r}^{p}f)^{**}(t) \leq\left(  \left[  \nabla_{r}^{p}f \right]  ^{p}
\right)  ^{**}(t)^{1/p}.
\]
Using again (\ref{eq:osc-drop}), we get
\[
O(f,t) \preceq\left(  \left[  \nabla_{r}^{p}f \right]  ^{p} \right)
^{**}(t)^{1/p}.
\]
Since $p_{0}=1$ and $r=v^{-1}(2t)$, this is precisely
(\ref{eq:symmetrization-gradient}).

\medskip

\noindent\textit{Proof of (ii).} Let $g\in D^{\phi}(f)$. By definition,
\[
|f(x)-f(y)| \leq\phi(d(x,y))\left(  g(x)+g(y)\right)
\]
for $\mu$-a.e. $x,y\in\Omega$.

Fix again $t>0$ and set
\[
r:=v^{-1}(2t).
\]
If $y\in B(x,r)$, then $d(x,y)\leq r$, and since $\phi$ is increasing,
\[
\phi(d(x,y))\leq\phi(r).
\]
Thus
\[
|f(x)-f(y)|^{p} \leq\phi(r)^{p}\left(  g(x)+g(y)\right)  ^{p}.
\]
Using
\[
(g(x)+g(y))^{p}\preceq g(x)^{p}+g(y)^{p},
\]
and averaging over $B(x,r)$, we obtain
\[
[\nabla_{r}^{p}f(x)]^{p} \preceq\phi(r)^{p}\left(  g(x)^{p}+A_{r}%
(g^{p})(x)\right)  .
\]
As in (\ref{eq:media}),
\[
A_{r}(g^{p})(x)\leq(g^{p})^{**}(2t).
\]
Therefore
\[
[\nabla_{r}^{p}f(x)]^{p} \preceq\phi(r)^{p}\left(  g(x)^{p}+(g^{p}%
)^{**}(2t)\right)  .
\]
Integrating the previous inequality over $E$ with $\mu(E)=t$, and taking the
supremum over all such $E$, we get
\[
\left(  [\nabla_{r}^{p}f]^{p}\right)  ^{**}(t) \preceq\phi(r)^{p}(g^{p}%
)^{**}(t).
\]
Combining this estimate with (\ref{eq:symmetrization-gradient}), we obtain
\[
O(|f|^{p_{0}},t)^{1/p_{0}} \preceq\phi(v^{-1}(2t))(g^{p})^{**}(t)^{1/p}.
\]
Since $v,\phi\in\mathcal{A}_{0}$, the function $\phi\circ v^{-1}$ is stable
under fixed dilations. Hence
\[
\phi(v^{-1}(2t))\simeq\phi(v^{-1}(t)).
\]
Setting
\[
\omega(t):=\phi(v^{-1}(t)),\qquad t>0,
\]
we arrive at
\[
O(|f|^{p_{0}},t)^{1/p_{0}} \preceq\omega(t)(g^{p})^{**}(t)^{1/p},
\]
which proves (\ref{eq:symmetrization-hajlasz-root}). The equivalent form
(\ref{eq:symmetrization-hajlasz}) follows by raising both sides to the power
$p_{0}$.
\end{proof}

\medskip

The next result shows that the lower ball estimate is also encoded in the
oscillation inequalities. More precisely, the validity of the preceding
symmetrization estimates, uniformly over the relevant class of functions,
implies (\ref{eq:lower-ball-v}), up to constants.

\begin{theorem}
\label{teo1-converse} Let $(\Omega,d,\mu)$ be a metric measure space, let
$0<p<\infty$, set
\[
p_{0}:=\min\{1,p\},
\]
and let $v\in\mathcal{A}_{0}$.

\begin{enumerate}
\item Assume that there exists a constant $C>0$ such that, for every $f\in
L^{p_{0}}(\mu)+L^{\infty}(\mu)$ and every $t>0$,
\begin{equation}
O(|f|^{p_{0}},t)^{1/p_{0}} \leq C\left(  [\nabla_{v^{-1}(t)}^{p}f]^{p}\right)
^{\ast\ast}(t)^{1/p}. \label{c1}%
\end{equation}
Then there exists a constant $K>0$ such that
\[
\inf_{x\in\Omega}\mu(B(x,R))\geq K\,v(R),\qquad R>0.
\]

\item Let $\phi\in\mathcal{A}_{0}$ be increasing and concave, chosen as a
strictly increasing representative. Assume that there exists a constant $C>0$
such that, for every $f\in L^{p_{0}}(\mu)+L^{\infty}(\mu)$, every $g\in
D^{\phi}(f)$, and every $t>0$,
\begin{equation}
O(|f|^{p_{0}},t)^{1/p_{0}} \leq C\,\phi(v^{-1}(t))\left(  (g^{p})^{\ast\ast
}(t)\right)  ^{1/p}. \label{eq:ass2}%
\end{equation}
Then there exists a constant $K_{0}>0$ such that
\[
\inf_{x\in\Omega}\mu(B(x,R))\geq K_{0}\,v(R),\qquad R>0.
\]

\end{enumerate}
\end{theorem}

\begin{proof}
We shall use the following elementary observation in both parts of the proof.
Let $B=B(x_{0},a)$, $m=\mu(B)$, and let $u$ be supported in $B$ with
\[
\|u\|_{L^{\infty}}=A.
\]
Then $(|u|^{p_{0}})^{*}(s)=0$ for $s>m$, and hence
\[
(|u|^{p_{0}})^{**}(2m) = \frac1{2m}\int_{0}^{2m}(|u|^{p_{0}})^{*}(s)\,ds =
\frac1{2m}\int_{0}^{m}(|u|^{p_{0}})^{*}(s)\,ds \le\frac{A^{p_{0}}}{2}.
\]
On the other hand,
\[
\lim_{t\to0^{+}}(|u|^{p_{0}})^{**}(t) = \|u\|_{L^{\infty}}^{p_{0}} = A^{p_{0}%
}.
\]
Therefore
\[
\frac{A^{p_{0}}}{2} \le\lim_{t\to0^{+}}(|u|^{p_{0}})^{**}(t) - (|u|^{p_{0}%
})^{**}(2m).
\]
Using
\[
\frac{d}{dt}(|u|^{p_{0}})^{**}(t) = -\frac{O(|u|^{p_{0}},t)}{t},
\]
we get
\begin{equation}
\frac{A^{p_{0}}}{2} \le\int_{0}^{2m}O(|u|^{p_{0}},s)\,\frac{ds}{s}.
\label{eq:common-osc-integral}%
\end{equation}

\medskip

\noindent\textit{Proof of (i).} Fix $x_{0}\in\Omega$ and $R>0$. Set
\[
B:=B(x_{0},R),\qquad m:=\mu(B),
\]
and define
\[
u(y):=(R-d(x_{0},y))_{+}.
\]
Then $u$ is supported in $B$ and $\Vert u\Vert_{L^{\infty}}=R$. Applying
(\ref{eq:common-osc-integral}) with $A=R$, and then using (\ref{c1}), we
obtain
\begin{equation}
R^{p_{0}} \preceq C^{p_{0}}\int_{0}^{2m} \left(  \left(  [\nabla_{v^{-1}%
(s)}^{p}u]^{p}\right)  ^{\ast\ast}(s) \right)  ^{p_{0}/p} \,\frac{ds}{s}.
\label{eq:conv-main-1}%
\end{equation}
Now, we estimate the integrand. Since
\[
u(y)=(R-d(x_{0},y))_{+},
\]
one has
\[
|u(x)-u(y)|\leq d(x,y)\left(  \chi_{B}(x)+\chi_{B}(y)\right)  ,\qquad
x,y\in\Omega.
\]
Hence, for every $0<p<\infty$,
\[
|u(x)-u(y)|^{p} \preceq d(x,y)^{p}\left(  \chi_{B}(x)+\chi_{B}(y)\right)  .
\]
Therefore
\[
\lbrack\nabla_{v^{-1}(s)}^{p}u]^{p}(x) \preceq v^{-1}(s)^{p}\left(  \chi
_{B}(x)+A_{v^{-1}(s)}\chi_{B}(x)\right)  .
\]
Applying the $^{\ast\ast}$-operator and using subadditivity,
\[
\left(  \lbrack\nabla_{v^{-1}(s)}^{p}u]^{p}\right)  ^{\ast\ast}(s) \preceq
v^{-1}(s)^{p}\left[  (\chi_{B})^{\ast\ast}(s) + (A_{v^{-1}(s)}\chi_{B}%
)^{\ast\ast}(s) \right]  .
\]
Since $0\leq\chi_{B}\leq1$ and $0\leq A_{v^{-1}(s)}\chi_{B}\leq1$, we have
\[
(\chi_{B})^{\ast\ast}(s)\leq1,\qquad(A_{v^{-1}(s)}\chi_{B})^{\ast\ast}%
(s)\leq1.
\]
Hence
\begin{equation}
\left(  \lbrack\nabla_{v^{-1}(s)}^{p}u]^{p}\right)  ^{\ast\ast}(s) \preceq
v^{-1}(s)^{p}. \label{eq:conv-nabla-final-1}%
\end{equation}
Substituting (\ref{eq:conv-nabla-final-1}) into (\ref{eq:conv-main-1}), we
get
\[
R^{p_{0}} \preceq C^{p_{0}}\int_{0}^{2m}v^{-1}(s)^{p_{0}}\,\frac{ds}{s}.
\]
Since $v^{-1}\in\mathcal{A}_{0}$, by (\ref{eqAo}), we have
\[
\int_{0}^{2m}v^{-1}(s)^{p_{0}}\,\frac{ds}{s} \preceq v^{-1}(2m)^{p_{0}}.
\]
Thus
\[
R^{p_{0}}\preceq C^{p_{0}}v^{-1}(2m)^{p_{0}}.
\]
Hence there exists $c>0$, independent of $x_{0}$ and $R$, such that
\[
v^{-1}(2m)\geq cR.
\]
Applying $v$, and absorbing fixed dilations since $v\in\mathcal{A}_{0}$, we
obtain
\[
2m\geq v(cR)\simeq v(R).
\]
Consequently,
\[
\mu(B(x_{0},R))=m\geq Kv(R).
\]
Since $x_{0}\in\Omega$ and $R>0$ were arbitrary, this proves
\[
\inf_{x\in\Omega}\mu(B(x,R))\geq Kv(R),\qquad R>0.
\]

\medskip

\noindent\textit{Proof of (ii).} Fix $x_{0}\in\Omega$ and $a>0$. Set
\[
B:=B(x_{0},a),\qquad m:=\mu(B),
\]
and define
\[
u(y):=\phi\bigl((a-d(x_{0},y))_{+}\bigr),\qquad\rho:=\chi_{B}.
\]
Since $\phi$ is increasing, concave, and satisfies $\phi(0)=0$, it is
subadditive. Hence, for all $\alpha,\beta\ge0$,
\[
|\phi(\alpha)-\phi(\beta)| \le\phi(|\alpha-\beta|).
\]
On the other hand,
\[
\bigl|(a-d(x_{0},x))_{+}-(a-d(x_{0},y))_{+}\bigr| \le|d(x_{0},x)-d(x_{0},y)|
\le d(x,y).
\]
Combining these two estimates, we obtain
\[
|u(x)-u(y)| \le\phi(d(x,y)).
\]
Since $u$ vanishes outside $B(x_{0},a)$, this yields
\[
|u(x)-u(y)| \le\phi(d(x,y))\bigl(\rho(x)+\rho(y)\bigr), \qquad x,y\in\Omega.
\]
Hence
\[
\rho\in D^{\phi}(u).
\]

The function $u$ is supported in $B$ and
\[
\Vert u\Vert_{L^{\infty}}=\phi(a).
\]
Applying (\ref{eq:common-osc-integral}) with $A=\phi(a)$, and then using
(\ref{eq:ass2}) with $f=u$ and $g=\rho$, we obtain
\[
\phi(a)^{p_{0}} \preceq C^{p_{0}}\int_{0}^{2m} \phi(v^{-1}(s))^{p_{0}} \left(
(\rho^{p})^{\ast\ast}(s)\right)  ^{p_{0}/p} \,\frac{ds}{s}.
\]
Since $\rho^{p}=\rho$ and
\[
\rho^{\ast\ast}(s)=\min\left(  1,\frac{m}{s}\right)  \leq1,
\]
we get
\[
\phi(a)^{p_{0}} \preceq C^{p_{0}}\int_{0}^{2m}\phi(v^{-1}(s))^{p_{0}}%
\,\frac{ds}{s}.
\]
Since $\phi\circ v^{-1}\in\mathcal{A}_{0}$, by (\ref{eqAo})
\[
\int_{0}^{2m}\phi(v^{-1}(s))^{p_{0}}\,\frac{ds}{s} \preceq\phi(v^{-1}%
(2m))^{p_{0}}.
\]
Therefore
\[
\phi(a)^{p_{0}} \preceq C^{p_{0}}\phi(v^{-1}(2m))^{p_{0}}.
\]
Hence there exists $c>0$, independent of $x_{0}$ and $a$, such that
\[
\phi(v^{-1}(2m))\geq c\,\phi(a).
\]
Since $\phi$ is strictly increasing,
\[
v^{-1}(2m)\geq\phi^{-1}(c\,\phi(a)).
\]
Since $\phi^{-1}\in\mathcal{A}_{0}$, fixed dilations are controlled by the
fundamental indices. Hence there exists $c_{1}>0$, independent of $a$, such
that
\[
\phi^{-1}(c\,\phi(a))\ge c_{1} a .
\]
Therefore
\[
v^{-1}(2m)\ge c_{1} a .
\]
Applying $v$, and absorbing fixed dilations once more, we obtain
\[
2m\geq v(c_{1}a)\simeq v(a).
\]
Consequently,
\[
\mu(B(x_{0},a))=m\geq K_{0}v(a).
\]
Since $x_{0}\in\Omega$ and $a>0$ were arbitrary, this proves
\[
\inf_{x\in\Omega}\mu(B(x,R))\geq K_{0}v(R),\qquad R>0.
\]

\end{proof}

In the next section we use the preceding symmetrization estimates to obtain
embedding results for averaged Besov and Haj\l asz--Sobolev type spaces.

\section{Function spaces and oscillation reductions}

\label{sec:oscillation-reductions}

We now introduce the function spaces that will be used in the
embedding results. The purpose of the section is twofold. First, we
fix the precise homogeneous seminorms and non-homogeneous norms
associated with Haj\l asz gradients and with averaged moduli of
smoothness. Second, we show how the symmetrization estimates
obtained above reduce the corresponding embedding problems to
estimates for oscillation functionals.

In these definitions, the relevant smoothness quantity is measured directly in
the given rearrangement-invariant space $X$: either a Haj\l asz gradient, or
an averaged modulus of smoothness.

\begin{definition}
\label{def:embedding-spaces} Let $\phi\in\mathcal{A}_{0}$, let $X$ be a
rearrangement-invariant space on $\Omega$.

\begin{enumerate}
\item The homogeneous $\phi$--Haj\l asz--Sobolev space $\dot{M}^{\phi
,X}(\Omega)$ consists of all $f\in L^{0}(\mu)$ such that
\[
\Vert f\Vert_{\dot{M}^{\phi,X}}:=\inf_{g\in D^{\phi}(f)}\Vert g\Vert
_{X}<\infty.
\]
The corresponding non-homogeneous space is
\[
M^{\phi,X}(\Omega):=\dot{M}^{\phi,X}(\Omega)\cap X,
\]
with norm
\[
\Vert f\Vert_{M^{\phi,X}}:=\Vert f\Vert_{\dot{M}^{\phi,X}}+\Vert f\Vert_{X}.
\]

\item Let $0<p<\infty,$ $0<q\leq\infty$ and $0<\theta<\infty$. The homogeneous
averaged Besov space $\dot{\mathcal{B}}_{X,q}^{\theta,p}(\Omega)$ consists of
all $f\in L_{\mathrm{loc}}^{p}(\Omega)$ such that
\[
\Vert f\Vert_{\dot{\mathcal{B}}_{X,q}^{\theta,p}}:=\left(  \int_{0}^{\infty
}\left(  \frac{E_{X}^{p}(f,r)}{r^{\theta}}\right)  ^{q}\frac{dr}{r}\right)
^{1/q}<\infty,
\]
with the usual modification when $q=\infty$, namely
\[
\Vert f\Vert_{\dot{\mathcal{B}}_{X,\infty}^{\theta,p}}:=\sup_{r>0}\frac
{E_{X}^{p}(f,r)}{r^{\theta}}.
\]
The corresponding non-homogeneous space is
\[
\mathcal{B}_{X,q}^{\theta,p}(\Omega):=\dot{\mathcal{B}}_{X,q}^{\theta
,p}(\Omega)\cap X,
\]
with norm
\[
\Vert f\Vert_{\mathcal{B}_{X,q}^{\theta,p}}:=\Vert f\Vert_{\dot{\mathcal{B}%
}_{X,q}^{\theta,p}}+\Vert f\Vert_{X}.
\]

\end{enumerate}
\end{definition}

\begin{remark}
The spaces $M^{1,p}(\Omega)$ were introduced by Haj\l asz \cite{Ha1} and have
become one of the basic objects in analysis of Sobolev spaces on metric
measure spaces; see also \cite{HK}. When $X=L^{p}$, the space $M^{\phi,L^{p}%
}(\Omega)$ coincides with the Haj\l asz--Sobolev space $M^{\phi,p}(\Omega)$
considered in \cite{Mas} and \cite{LYY}. In the particular case
\[
\phi(t)=t^{s},\qquad t\ge0,
\]
we recover the fractional Haj\l asz--Sobolev space $M^{s,p}(\Omega)$.

In the Euclidean setting one has, for $p>1$,
\[
M^{1,p}(\mathbb{R}^{n})=W^{1,p}(\mathbb{R}^{n}),
\]
whereas the endpoint space $M^{1,1}(\mathbb{R}^{n})$ coincides with the
Hardy--Sobolev space $H^{1,1}(\mathbb{R}^{n})$; see \cite[Theorem 1]{KS}.
Moreover, for $0<s<1$,
\[
M^{s,p}(\mathbb{R}^{n})=B^{s}_{p,\infty}(\mathbb{R}^{n}),
\]
up to equivalence of quasi-norms; see \cite{Yang}. Notice also that, in
$\mathbb{R}^{n}$, if $s>1$, then $M^{s,p}(\mathbb{R}^{n})$ is trivial, while
on fractal sets such spaces may be non-trivial; see \cite{Hu}.

Finally, the Haj\l asz--Besov scale defined through the averaged modulus
$E_{L^{p}}^{p}(f,r)$ agrees, in the Euclidean setting, with the usual Besov
scale; see Remark~\ref{rem1}. Thus
\[
\dot{\mathcal{B}}_{L^{p},q}^{\theta,p}(\mathbb{R}^{n}) = \dot B^{\theta}%
_{p,q}(\mathbb{R}^{n}),
\]
up to equivalence of quasi-norms.
\end{remark}

%------------

\subsection{Reduction to oscillation estimates}

\begin{proposition}
\label{prop:oscillation-reductions} Let $X$ be an r.i. space and let
$\phi,v\in\mathcal{A}_{0}$. Assume that
\[
h_{\mu}(r):=\inf_{x\in\Omega}\mu(B(x,r))\ge v(r),\qquad r>0.
\]

\begin{enumerate}
\item[(i)] Choose $0<\eta=\eta(X)\le1$ such that
\[
\overline{\alpha}_{X}<\frac1\eta.
\]
Then
\[
\left\Vert \frac{O(|f|^{\eta},\cdot)^{1/\eta}}{\omega_{H}(\cdot)} \right\Vert
_{X} \preceq\Vert f\Vert_{\dot M^{\phi,X}},
\]
where
\[
\omega_{H}(t):=\phi(v^{-1}(t)),\qquad t>0.
\]
Notice that such an $\eta$ can always be chosen with $0<\eta<1$; if
$\overline{\alpha}_{X}<1$, one may take $\eta=1$.

\item[(ii)] Let $0<p<\infty$, set
\[
p_{0}:=\min\{1,p\},
\]
let $0<q\le\infty$, and let $0<\theta<\infty$. Then
\[
\left(  \int_{0}^{\infty} \left(  \frac{\varphi_{X}(t)O(|f|^{p_{0}%
},t)^{1/p_{0}}} {v^{-1}(t)^{\theta}} \right)  ^{q} \frac{dt}{t} \right)
^{1/q} \preceq\Vert f\Vert_{\dot{\mathcal{B}}_{X,q}^{\theta,p}},
\]
with the usual modification when $q=\infty$.
\end{enumerate}
\end{proposition}

\begin{proof}
We prove the two assertions separately.

\medskip

\noindent\textit{Proof of (i).} Let $g\in D^{\phi}(f)$. By
Theorem~\ref{thm:symmetrization}, applied with the exponent $\eta$, we get
\[
O(|f|^{\eta},t)^{1/\eta}\preceq\omega_{H}(t)(g^{\eta})^{\ast\ast}(t)^{1/\eta
}.
\]
Therefore
\[
\left\Vert \frac{O(|f|^{\eta},\cdot)^{1/\eta}}{\omega_{H}(\cdot)}\right\Vert
_{X}\preceq\left\Vert \left(  (g^{\eta})^{\ast\ast}\right)  ^{1/\eta
}\right\Vert _{X}.
\]
Since
\[
\overline{\alpha}_{X}<\frac{1}{\eta},
\]
(\ref{Boydsup}) implies the Hardy operator $P^{(\eta)}$ is bounded on $X$.
Hence
\[
\left\Vert \left(  (g^{\eta})^{\ast\ast}\right)  ^{1/\eta}\right\Vert
_{X}\preceq\Vert g\Vert_{X}.
\]
Taking the infimum over $g\in D^{\phi}(f)$, we obtain
\[
\left\Vert \frac{O(|f|^{\eta},\cdot)^{1/\eta}}{\omega_{H}(\cdot)}\right\Vert
_{X}\preceq\Vert f\Vert_{\dot{M}^{\phi,X}}.
\]

\medskip

\noindent\textit{Proof of (ii).} Since $0<p_{0}\leq1$, we may apply
Theorem~\ref{thm:symmetrization} with exponent $p_{0}$. Thus
\[
O(|f|^{p_{0}},t)\preceq\left(  \left[  \nabla_{v^{-1}(2t)}^{p_{0}}f\right]
^{p_{0}}\right)  ^{\ast\ast}(t).
\]
By H\"{o}lder inequality (\ref{Hold}) and (\ref{si}) we obtain the estimate
\[
h^{\ast\ast}(t)=\frac{1}{t}\int_{0}^{t}h^{\ast}(s)ds\leq\Vert h\Vert_{Y}%
\frac{\varphi_{Y^{\prime}}(t)}{t}=\frac{\Vert h\Vert_{Y}}{\varphi_{Y}%
(t)},\qquad t>0,
\]
Applying it with
\[
h=\left[  \nabla_{v^{-1}(2t)}^{p_{0}}f\right]  ^{p_{0}}\qquad\text{and}\qquad
Y=X^{(1/p_{0})},
\]
we get
\[
\left(  \left[  \nabla_{v^{-1}(2t)}^{p_{0}}f\right]  ^{p_{0}}\right)
^{\ast\ast}(t)\leq\frac{\left\Vert \left[  \nabla_{v^{-1}(2t)}^{p_{0}%
}f\right]  ^{p_{0}}\right\Vert _{X^{(1/p_{0})}}}{\varphi_{X^{(1/p_{0})}}(t)}.
\]
Since
\[
\left\Vert \left[  \nabla_{v^{-1}(2t)}^{p_{0}}f\right]  ^{p_{0}}\right\Vert
_{X^{(1/p_{0})}}=E_{X}^{p_{0}}(f,v^{-1}(2t))^{p_{0}}%
\]
and
\[
\varphi_{X^{(1/p_{0})}}(t)=\varphi_{X}(t)^{p_{0}},
\]
we obtain
\[
\left(  \left[  \nabla_{v^{-1}(2t)}^{p_{0}}f\right]  ^{p_{0}}\right)
^{\ast\ast}(t)\leq\frac{E_{X}^{p_{0}}(f,v^{-1}(2t))^{p_{0}}}{\varphi
_{X}(t)^{p_{0}}}.
\]
Consequently,
\begin{equation}
\varphi_{X}(t)O(|f|^{p_{0}},t)^{1/p_{0}}\preceq E_{X}^{p_{0}}(f,v^{-1}(2t)).
\label{eq:besov-pointwise-oscillation}%
\end{equation}

Since $p_{0}\leq p$, Jensen's inequality gives
\[
E_{X}^{p_{0}}(f,r)\leq E_{X}^{p}(f,r),\qquad r>0.
\]
Therefore, from (\ref{eq:besov-pointwise-oscillation}),%
\[
\varphi_{X}(t)O(|f|^{p_{0}},t)^{1/p_{0}}\preceq E_{X}^{p}(f,v^{-1}(2t)).
\]

Dividing by $v^{-1}(2t)^{\theta}$, taking the $L^{q}((0,\infty),dt/t)$%
-quasi-norm, and using the preceding estimate, we obtain
\[
\left(  \int_{0}^{\infty}\left(  \frac{\varphi_{X}(t)O(|f|^{p_{0}}%
,t)^{1/p_{0}}} {v^{-1}(2t)^{\theta}} \right)  ^{q} \frac{dt}{t} \right)
^{1/q} \preceq\left(  \int_{0}^{\infty}\left(  \frac{E_{X}^{p}(f,v^{-1}(2t))}
{v^{-1}(2t)^{\theta}} \right)  ^{q} \frac{dt}{t} \right)  ^{1/q}.
\]
Using the representative of $v$ fixed in Remark~\ref{rem0}, we have
\[
uv^{\prime}(u)\simeq v(u),\qquad u>0.
\]
Thus the change of variables $2t=v(u)$ gives
\[
\left(  \int_{0}^{\infty}\left(  \frac{E_{X}^{p}(f,v^{-1}(2t))} {v^{-1}%
(2t)^{\theta}} \right)  ^{q} \frac{dt}{t} \right)  ^{1/q} \simeq\left(
\int_{0}^{\infty}\left(  \frac{E_{X}^{p}(f,u)}{u^{\theta}} \right)  ^{q}
\frac{du}{u} \right)  ^{1/q}.
\]
Since $v^{-1}(2t)\simeq v^{-1}(t)$, the left-hand side is equivalent to
\[
\left(  \int_{0}^{\infty}\left(  \frac{\varphi_{X}(t)O(|f|^{p_{0}}%
,t)^{1/p_{0}}} {v^{-1}(t)^{\theta}} \right)  ^{q} \frac{dt}{t} \right)
^{1/q}.
\]
Therefore
\[
\left(  \int_{0}^{\infty}\left(  \frac{\varphi_{X}(t)O(|f|^{p_{0}}%
,t)^{1/p_{0}}} {v^{-1}(t)^{\theta}} \right)  ^{q} \frac{dt}{t} \right)  ^{1/q}
\preceq\Vert f\Vert_{\dot{\mathcal{B}}_{X,q}^{\theta,p}}.
\]
This proves the second assertion. The case $q=\infty$ is obtained in the same
way, replacing integrals by suprema.
\end{proof}

\subsection{Normed oscillation inequalities and lower growth}

For simplicity and clarity of presentation, the normed oscillation
equivalences in this section are formulated in the power case
\[
h_{\mu}(r)\succeq r^{Q}.
\]
The argument is not specific to power growth. The same proof applies to a
general admissible lower growth function $v$, after replacing the power
weights by the corresponding weights determined by $v$. We restrict ourselves
to the power setting in order to avoid additional notation and to make the
resulting function spaces explicit. We now show that the normed oscillation
inequalities obtained above retain the lower geometric information encoded by
the measure. To make this fact explicit and to connect the abstract results
with the model scale considered in the next section, we specialize, up to
multiplicative constants, to
\begin{equation}
v(r)\simeq r^{Q}, \qquad\phi(r)=r^{\alpha}, \qquad0<\alpha\le1.
\label{eq:power-model-functions}%
\end{equation}
Consequently,
\begin{equation}
\phi(v^{-1}(t)) \simeq t^{\alpha/Q}. \label{eq:power-model-weight}%
\end{equation}

\begin{theorem}
\label{thm:normed-oscillation-lower-growth} Let $Q>0$,
$0<\alpha\le1$, and let $X$ be an r.i. space satisfying
$\overline{\alpha}_{X}<1. $ Then the following assertions are
equivalent.

\begin{enumerate}
\item For every $f\in M^{\alpha,X}$,
\[
\left\|  t^{-\alpha/Q}O(f,t) \right\|  _{X} \preceq\|f\|_{\dot M^{\alpha,X}}.
\]

\item The measure satisfies
\[
h_{\mu}(r)\succeq r^{Q}, \qquad r>0.
\]

\end{enumerate}
\end{theorem}

\begin{proof}
Assume first that assertion $(i)$ holds. Fix $x\in\Omega$ and $r>0$, and
introduce the following dyadic cut-off construction.\footnote{The construction
below and the subsequent iteration are inspired by the dyadic cut-off argument
in \cite[Construction~14 and Lemma~16]{AGH}; here they are adapted to the
normed oscillation inequality.} For $j=0,1,2,\ldots$, set
\begin{equation}
\begin{gathered} r_j=\frac r2+\frac{r}{2^{j+1}}, \qquad B_j=B(x,r_j), \qquad a_j=\mu(B_j), \\ \delta_j=r_j-r_{j+1}=\frac{r}{2^{j+2}}, \qquad u_j(y)=\left( 1-\frac{\operatorname{dist}(y,B_{j+1})}{\delta_j} \right)_+. \end{gathered} \label{eq:dyadic-cutoffs}%
\end{equation}
Then
\begin{equation}
u_{j}=1 \quad\text{on }B_{j+1}, \qquad\operatorname{supp}u_{j}\subset B_{j},
\qquad0\le u_{j}\le1, \label{eq:dyadic-cutoff-properties}%
\end{equation}
and $u_{j}$ is $\delta_{j}^{-1}$-Lipschitz.

Set
\[
g_{j}:=\delta_{j}^{-\alpha}\chi_{B_{j}}.
\]
We claim that $g_{j}\in D^{\alpha}(u_{j})$. Indeed, the distance
function to a set is $1$-Lipschitz, and hence $u_{j}$ is
$\delta_{j}^{-1}$-Lipschitz. Together with $0\le u_{j}\le1$, this
gives
\[
|u_{j}(z)-u_{j}(y)| \le\min\left\{  1,\frac{d(z,y)}{\delta_{j}} \right\}  .
\]
Since $0<\alpha\le1$ and $\min{(1,t)}\le t^{\alpha}$ for every
$t\ge0$, we conclude that
\[
|u_{j}(z)-u_{j}(y)| \le\left(  \frac{d(z,y)}{\delta_{j}} \right)  ^{\alpha}.
\]
If $u_{j}(z)\ne u_{j}(y)$, then at least one of the two values is
nonzero. Suppose, for instance, that $u_{j}(z)>0$. Then
\[
\operatorname{dist}(z,B_{j+1})<\delta_{j}.
\]
Since $r_{j+1}+\delta_{j}=r_{j}$, this implies $z\in B_{j}$. Thus
\[
\chi_{B_{j}}(z)+\chi_{B_{j}}(y)\ge1,
\]
and consequently
\[
\begin{aligned}
|u_j(z)-u_j(y)| &\le \left( \frac{d(z,y)}{\delta_j} \right)^\alpha
\\
&\le d(z,y)^\alpha \bigl(g_j(z)+g_j(y)\bigr).
\end{aligned}
\]
If $u_{j}(z)=u_{j}(y)$, the inequality is trivial. Hence $g_{j}\in
D^{\alpha}(u_{j})$.

We next estimate the oscillation of $u_{j}$. By
\eqref{eq:dyadic-cutoff-properties},
\[
u_{j}^{*}(t)=0, \qquad t>a_{j},
\]
whereas
\[
u_{j}^{*}(s)=1, \qquad0<s<a_{j+1}.
\]
Consequently, for $a_{j}<t<2a_{j}$,
\begin{equation}
O(u_{j},t) = u_{j}^{**}(t) \ge\frac{a_{j+1}}{t}.
\label{eq:dyadic-cutoff-oscillation}%
\end{equation}
Thus
\[
t^{-\alpha/Q}O(u_{j},t) \ge2^{-1-\alpha/Q} a_{j+1}a_{j}^{-1-\alpha/Q}
\chi_{(a_{j},2a_{j})}(t),
\]
and then
\[
\begin{aligned}
\left\| t^{-\alpha/Q}O(u_j,t) \right\|_X &\succeq
a_{j+1}a_j^{-1-\alpha/Q} \|\chi_{(a_j,2a_j)}\|_X
\\
&= a_{j+1}a_j^{-1-\alpha/Q} \varphi_X(a_j).
\end{aligned}
\]
Since
\[
\|u_{j}\|_{\dot M^{\alpha,X}} \le\|g_{j}\|_{X} = \delta_{j}^{-\alpha}%
\varphi_{X}(a_{j}),
\]
assertion $(i)$ gives
\[
a_{j+1}a_{j}^{-1-\alpha/Q}\varphi_{X}(a_{j}) \preceq\delta_{j}^{-\alpha
}\varphi_{X}(a_{j}).
\]
Cancelling the fundamental function, we obtain
\begin{equation}
a_{j+1} \preceq\delta_{j}^{-\alpha}a_{j}^{1+\alpha/Q}.
\label{eq:normed-oscillation-recursive}%
\end{equation}

Set
\[
\gamma:=\frac{Q}{Q+\alpha}.
\]
Then
\[
0<\gamma<1, \qquad\gamma\left(  1+\frac{\alpha}{Q}\right)  =1.
\]
Writing \eqref{eq:normed-oscillation-recursive} as
\[
a_{j+1} \le C\delta_{j}^{-\alpha}a_{j}^{1+\alpha/Q},
\]
and raising both sides to the power $\gamma$, we obtain
\[
a_{j+1}^{\gamma}\le C^{\gamma}\delta_{j}^{-\alpha\gamma}a_{j}.
\]
Therefore,
\[
\begin{aligned}
a_j &\ge C^{-\gamma}\delta_j^{\alpha\gamma}a_{j+1}^\gamma
\\
&= C^{-\gamma} r^{\alpha\gamma} 2^{-(j+2)\alpha\gamma}
a_{j+1}^\gamma.
\end{aligned}
\]
Absorbing the factor $C^{-\gamma}2^{-2\alpha\gamma}$ into a constant $c_{0}%
>0$, we get
\[
a_{j} \ge c_{0} r^{\alpha\gamma} 2^{-j\alpha\gamma} a_{j+1}^{\gamma},
\]
where $c_{0}$ is independent of $j$, $x$, and $r$.

We spell out the first steps of the iteration. For $j=1$,
\[
a_{1} \ge c_{0} r^{\alpha\gamma} 2^{-\alpha\gamma} a_{2}^{\gamma}.
\]
Applying the recursive estimate to $a_{2}$, we obtain
\[
\begin{aligned}
a_1 &\ge c_0 r^{\alpha\gamma} 2^{-\alpha\gamma} \left( c_0
r^{\alpha\gamma} 2^{-2\alpha\gamma} a_3^\gamma \right)^\gamma
\\
&= c_0^{1+\gamma} r^{\alpha\gamma(1+\gamma)}
2^{-\alpha\gamma(1+2\gamma)} a_3^{\gamma^2}.
\end{aligned}
\]
Applying it once more gives
\[
\begin{aligned}
a_1 &\ge c_0^{1+\gamma} r^{\alpha\gamma(1+\gamma)}
2^{-\alpha\gamma(1+2\gamma)} \left( c_0 r^{\alpha\gamma}
2^{-3\alpha\gamma} a_4^\gamma \right)^{\gamma^2}
\\
&= c_0^{1+\gamma+\gamma^2} r^{\alpha\gamma(1+\gamma+\gamma^2)}
2^{-\alpha\gamma(1+2\gamma+3\gamma^2)} a_4^{\gamma^3}.
\end{aligned}
\]
Continuing in this way, after $N$ steps,
\[
a_{1} \ge c_{0}^{S_{N}} r^{\alpha\gamma S_{N}} 2^{-\alpha\gamma T_{N}}
a_{N+1}^{\gamma^{N}},
\]
where
\[
S_{N} := \sum_{k=0}^{N-1}\gamma^{k}, \qquad T_{N} := \sum_{k=1}^{N}%
k\gamma^{k-1}.
\]

Since
\[
B(x,r/2) \subset B_{N+1} \subset B_{1},
\]
and balls have positive and finite measure,
\[
a_{N+1}^{\gamma^{N}} \longrightarrow1.
\]
Moreover,
\[
S_{N} \longrightarrow\frac1{1-\gamma}, \qquad T_{N} \longrightarrow
\frac1{(1-\gamma)^{2}}.
\]
Letting $N\to\infty$, we obtain
\[
a_{1} \succeq r^{\alpha\gamma/(1-\gamma)}.
\]
Since
\[
\frac{\alpha\gamma}{1-\gamma} = Q,
\]
it follows that
\[
a_{1} \succeq r^{Q}.
\]
Finally,
\[
a_{1} = \mu\left(  B\left(  x,\frac{3r}{4}\right)  \right)  \le\mu(B(x,r)).
\]
Taking the infimum over $x\in\Omega$, we conclude that
\[
h_{\mu}(r) \succeq r^{Q}, \qquad r>0.
\]
This proves assertion $(ii)$.

Conversely, assertion $(ii)$, together with
Proposition~\ref{prop:oscillation-reductions}(i) applied with $\eta=1$, gives
assertion $(i)$.
\end{proof}

\begin{remark}
The restriction
\[
\overline{\alpha}_{X}<1
\]
is used only to work with oscillation exponent $1$. The endpoint case
$\overline{\alpha}_{X}=1$ can be treated by choosing any $0<\eta<1$ and
replacing $O(f,t)$ by
\[
O(|f|^{\eta},t)^{1/\eta}.
\]
Indeed,
\[
\overline{\alpha}_{X^{(1/\eta)}} = \eta\,\overline{\alpha}_{X} = \eta<1,
\]
so the Hardy averaging operator is bounded on the convexified space
$X^{(1/\eta)}$. Consequently, the lower ball estimate
\[
h_{\mu}(r)\succeq r^{Q}
\]
implies
\[
\left\|  t^{-\alpha/Q} O(|f|^{\eta},t)^{1/\eta} \right\|  _{X} \preceq
\|f\|_{\dot M^{\alpha,X}}.
\]

Conversely, using the cutoffs defined in \eqref{eq:dyadic-cutoffs}, the same
argument gives
\[
a_{j+1} \preceq\delta_{j}^{-\alpha\eta} a_{j}^{1+\alpha\eta/Q}.
\]
Taking
\[
\gamma= \frac{Q}{Q+\alpha\eta},
\]
the preceding iteration again yields
\[
a_{1} \succeq r^{Q}.
\]
Thus the equivalence remains valid when $\overline{\alpha}_{X}=1$, provided
the powered oscillation functional is used.
\end{remark}

The corresponding result for the averaged Besov scale is as follows.

\begin{theorem}
\label{thm:normed-besov-lower-growth} Let $Q>0$, $0<\theta<1$, $0<p<\infty$,
and $0<q\le\infty$, and set
\[
p_{0}:=\min\{1,p\}.
\]
Let $X$ be an r.i. space such that $X^{(1/p)}$ is an r.i. space.
Assume moreover that
\[
D_{X,p} := \sup_{s>0} \|A_{s}\|_{X^{(1/p)}\to X^{(1/p)}} < \infty.
\]
Then the following assertions are equivalent.

\begin{enumerate}
\item For every measurable function $f$,
\[
\left(  \int_{0}^{\infty}\left[  \varphi_{X}(t)t^{-\theta/Q} O(|f|^{p_{0}%
},t)^{1/p_{0}} \right]  ^{q} \frac{dt}{t} \right)  ^{1/q} \preceq
\|f\|_{\dot{\mathcal{B}}_{X,q}^{\theta,p}},
\]
with the usual modification when $q=\infty$.

\item The measure satisfies
\[
h_{\mu}(r)\succeq r^{Q}, \qquad r>0.
\]

\end{enumerate}
\end{theorem}

\begin{proof}
Assume first that assertion $(i)$ holds. Fix $x\in\Omega$ and $r>0$, and
consider the quantities $r_{j},\, B_{j},\, a_{j},\, \delta_{j},\, u_{j} $
defined in \eqref{eq:dyadic-cutoffs}.

By \eqref{eq:dyadic-cutoff-properties}, for $a_{j}<t<2a_{j}$,
\[
O(|u_{j}|^{p_{0}},t) \ge\frac{a_{j+1}}{t}.
\]
Since $\varphi_{X}$ is increasing, it follows that
\[
\begin{aligned}
& \left( \int_0^\infty \left[ \varphi_X(t)t^{-\theta/Q}
O(|u_j|^{p_0},t)^{1/p_0} \right]^q \frac{dt}{t} \right)^{1/q}
\\
&\qquad\succeq \varphi_X(a_j) a_{j+1}^{1/p_0} a_j^{-1/p_0-\theta/Q},
\end{aligned}
\]
with the same conclusion when $q=\infty$.

We now estimate the averaged Besov seminorm. The support and Lipschitz
properties in \eqref{eq:dyadic-cutoff-properties} imply
\[
[\nabla_{s}^{p} u_{j}(z)]^{p} \preceq\min\left\{  1,\frac{s^{p}}{\delta
_{j}^{p}} \right\}  \left[  \chi_{B_{j}}(z)+A_{s}\chi_{B_{j}}(z) \right]  .
\]
Therefore, by the definition of $D_{X,p}$,
\[
\begin{aligned}
E_X^p(u_j,s)^p &\preceq \min\left\{ 1,\frac{s^p}{\delta_j^p}
\right\} \|\chi_{B_j}\|_{X^{(1/p)}}
\\
&= \min\left\{ 1,\frac{s^p}{\delta_j^p} \right\} \varphi_X(a_j)^p.
\end{aligned}
\]
Hence
\[
E_{X}^{p}(u_{j},s) \preceq\min\left\{  1,\frac{s}{\delta_{j}} \right\}
\varphi_{X}(a_{j}).
\]
Since $0<\theta<1$,
\[
\|u_{j}\|_{\dot{\mathcal{B}}_{X,q}^{\theta,p}} \preceq\delta_{j}^{-\theta
}\varphi_{X}(a_{j}),
\]
with the usual modification when $q=\infty$.

Applying assertion $(i)$ to $u_{j}$ and cancelling $\varphi_{X}(a_{j})$, we
obtain
\[
a_{j+1}^{1/p_{0}} \preceq\delta_{j}^{-\theta} a_{j}^{1/p_{0}+\theta/Q}.
\]
Equivalently,
\begin{equation}
a_{j+1} \preceq\delta_{j}^{-\theta p_{0}} a_{j}^{1+\theta p_{0}/Q}.
\label{eq:normed-besov-recursive}%
\end{equation}

Set
\[
\gamma:= \frac{Q}{Q+\theta p_{0}}.
\]
Since $\delta_{j}=r/2^{j+2}$, \eqref{eq:normed-besov-recursive} gives
\[
a_{j} \ge c_{0} r^{\theta p_{0}\gamma} 2^{-j\theta p_{0}\gamma} a_{j+1}%
^{\gamma},
\]
where $c_{0}>0$ is independent of $j$, $x$, and $r$.

The iteration carried out in the proof of
Theorem~\ref{thm:normed-oscillation-lower-growth}, with $\alpha$ replaced by
$\theta p_{0}$, yields
\[
a_{1} \succeq r^{\theta p_{0}\gamma/(1-\gamma)} = r^{Q}.
\]
Since
\[
a_{1} = \mu\left(  B\left(  x,\frac{3r}{4}\right)  \right)  \le\mu(B(x,r)),
\]
taking the infimum over $x\in\Omega$ proves assertion $(ii)$.

Conversely, assertion $(ii)$, together with
Proposition~\ref{prop:oscillation-reductions}(ii), gives assertion $(i)$.
\end{proof}

\begin{remark}
The assumption that $X^{(1/p)}$ be an r.i. space is automatic when $0<p\le1$,
since $1/p\ge1$. For $p>1$, it is an additional convexity requirement; it is
satisfied, for instance, when $X$ is $p$-convex, after passing to an
equivalent lattice norm if necessary; see \cite{LT2}.

Moreover, suppose that
\[
D:=\sup_{s>0}\|A_{s}\|_{L^{1}\to L^{1}}<\infty.
\]
Since
\[
\|A_{s}\|_{L^{\infty}\to L^{\infty}}\le1,
\]
the standard interpolation theorem for rearrangement-invariant spaces (see
\cite{BS,KPS}) yields
\[
\sup_{s>0} \|A_{s}\|_{X^{(1/p)}\to X^{(1/p)}}<\infty
\]
whenever $X^{(1/p)}$ is a Banach function space. Thus, in this setting, the
uniform $L^{1}$-boundedness of the averaging operators implies the operator
assumption used in the theorem.
\end{remark}

\section{Oscillation targets and classical Sobolev embeddings}

\label{sec:embedding-consequences}

The results of the preceding section show that, in the power case, the lower
ball estimate is equivalent to an embedding into a space defined by a normed
oscillation functional. This is the level at which the geometric information
is retained.

This point of view connects our results with the literature on converse
Sobolev embeddings. Converse results starting from prescribed Sobolev,
Haj\l asz, Besov, or fractional Sobolev targets have been obtained in
\cite{GorkaPA,AGH,Karak,GorkaNA,GSStudia}. In the subcritical range, the
oscillation spaces can be identified with classical Lorentz spaces. This
recovers the usual Sobolev embeddings and, at the same time, yields the finer
Lorentz targets.

More generally, the cutoff argument developed below shows directly what lower
ball growth is forced by an arbitrary rearrangement-invariant target. It also
clarifies why, in the critical and supercritical regimes, the final global
target no longer retains all the geometric information. In the critical case,
one is naturally led to localized Trudinger inequalities, whereas in the
supercritical case the sharp formulation lies outside the
rearrangement-invariant framework and takes the form of a H\"older--Morrey estimate.

Classical Sobolev targets are obtained by reconstructing the
decreasing rearrangement of a function from its oscillation and
applying an appropriate Hardy operator; see
\cite{BS,JMosc,PA,MO,MM6,MM3}. We illustrate this mechanism in the
power-growth setting.

Let $0<p_{0}\leq1$, and assume that $f^{\ast\ast}(\infty)=0$. By
\eqref{der2est},
\[
\bigl(|f|^{p_{0}}\bigr)^{\ast\ast}(t)=\int_{t}^{\infty}O(|f|^{p_{0}}%
,s)\,\frac{ds}{s}.
\]
Therefore,we have
\[
t^{-\alpha/Q}\bigl((|f|^{p_{0}})^{\ast\ast}(t)\bigr)^{1/p_{0}}=Q_{\lambda
}^{(p_{0})}\left(  \left(  \frac{O(|f|^{p_{0}},\cdot)}{\left(  \cdot\right)
^{\alpha p_{0}/Q}}\right)  ^{1/p_{0}}\right)  (t).
\]
with $\lambda=\frac{\alpha p_{0}}{Q}.$

If $X$ is an r.i. space such that $\underline{\alpha}%
_{X}>\frac{\alpha}{Q},$ then $Q_{\lambda}^{(p_{0})}$ is bounded on $X$ by
\eqref{Boydinf}. Consequently,
\[
\left\Vert t^{-\alpha/Q}\bigl((|f|^{p_{0}})^{\ast\ast}(t)\bigr)^{1/p_{0}%
}\right\Vert _{X}\preceq\left\Vert \left(  \frac{O(|f|^{p_{0}},t)}{t^{\alpha
p_{0}/Q}}\right)  ^{1/p_{0}}\right\Vert _{X}.
\]
The reverse estimate follows from
\[
O(|f|^{p_{0}},t)\leq\bigl(|f|^{p_{0}}\bigr)^{\ast\ast}(t).
\]
Thus,
\[
\left\Vert \left(  \frac{O(|f|^{p_{0}},t)}{t^{\alpha p_{0}/Q}}\right)
^{1/p_{0}}\right\Vert _{X}\simeq\left\Vert t^{-\alpha/Q}\bigl((|f|^{p_{0}%
})^{\ast\ast}(t)\bigr)^{1/p_{0}}\right\Vert _{X}.
\]
Moreover, since \(0<p_0\leq1\), it follows from
\cite[Theorem~4.5]{Tur} that
\[
\left\| t^{-\alpha/Q} \bigl((|f|^{p_0})^{**}(t)\bigr)^{1/p_0}
\right\|_X \simeq \left\| t^{-\alpha/Q}f^{**}(t) \right\|_X.
\]
Consequently,
\[
\left\| t^{-\alpha/Q}O(|f|^{p_0},t)^{1/p_0} \right\|_X \simeq
\left\| t^{-\alpha/Q}f^{**}(t) \right\|_X.
\]

This identification immediately yields the classical Lorentz
targets. Recall that, if \(1<r<\infty\), \(0<q\leq\infty\), and
\(f^*(\infty)=0\), then
\[
\|f\|_{L^{r,q}} \simeq \left( \int_0^\infty
\left[t^{1/r}O(f,t)\right]^q \frac{dt}{t} \right)^{1/q},
\]
with the usual modification when \(q=\infty\); see
\cite{BS,BMR,JMosc}.

\begin{corollary}[Lorentz embeddings]\label{cor:lorentz-embeddings}
The following equivalences hold.

\begin{enumerate}
\item
Let \(0<\alpha\leq1\) and
\[
1<p<\frac{Q}{\alpha}, \qquad p_\alpha^*:=\frac{Qp}{Q-\alpha p}.
\]
Then
\[
h_\mu(r)\succeq r^Q \quad\Longleftrightarrow\quad
M^{\alpha,p}(\Omega) \hookrightarrow L^{p_\alpha^*,p}(\Omega).
\]

\item
Let \(0<\theta<1\),
\[
1\leq p<\frac{Q}{\theta}, \qquad 0<q\leq\infty, \qquad
p_\theta^*:=\frac{Qp}{Q-\theta p},
\]
and assume that
\[
\sup_{s>0}\|A_s\|_{L^1\to L^1}<\infty.
\]
Then
\[
h_\mu(r)\succeq r^Q \quad\Longleftrightarrow\quad \mathcal
B_{L^p,q}^{\theta,p}(\Omega) \hookrightarrow
L^{p_\theta^*,q}(\Omega).
\]
\end{enumerate}
\end{corollary}

\begin{proof}
The first assertion follows from the preceding identification with
\(X=L^p\), since
\[
\frac1{p_\alpha^*} = \frac1p-\frac{\alpha}{Q}.
\]
For the averaged Besov scale,
\[
\frac1{p_\theta^*} = \frac1p-\frac{\theta}{Q},
\]
and the oscillation functional in
Theorem~\ref{thm:normed-besov-lower-growth} is equivalent to the
\(L^{p_\theta^*,q}\)-norm. The conclusion follows from the
oscillation characterization of Lorentz spaces.
\end{proof}

This mechanism requires no doubling, Poincar\'e, Ahlfors regularity,
or uniform perfectness assumption. For the averaged Besov scale, the
only additional hypothesis is the boundedness of the averaging
operators required in Theorem~\ref{thm:normed-besov-lower-growth}.

\subsection{Lower growth forced by rearrangement-invariant targets}
We next determine the lower ball growth forced by embeddings into
general rearrangement-invariant target spaces.
\begin{proposition}
\label{prop:embedding-lower-growth} Let $X$ and $Y$ be r.i. spaces,
and let $0<\alpha\le1$. Assume that
\begin{equation}
\|f\|_{Y} \preceq\|f\|_{\dot M^{\alpha,X}} \label{eq:ri-target-embedding}%
\end{equation}
for every boundedly supported Lipschitz function $f$. Then, with the notation
introduced in \eqref{eq:dyadic-cutoffs},
\begin{equation}
\varphi_{Y}(a_{j+1}) \preceq\delta_{j}^{-\alpha}\varphi_{X}(a_{j}), \qquad
j=0,1,2,\ldots, \label{eq:ri-target-first-recurrence}%
\end{equation}
and consequently
\begin{equation}
a_{j} \succeq\delta_{j}^{\alpha}a_{j+1} \frac{\varphi_{Y}(a_{j+1})}
{\varphi_{X}(a_{j+1})}. \label{eq:ri-target-second-recurrence}%
\end{equation}

In particular, if there exists $0<\beta\le1$ such that
\begin{equation}
\varphi_{Y}(t) \succeq t^{-\beta}\varphi_{X}(t), \qquad t>0,
\label{eq:fundamental-power-gap}%
\end{equation}
then
\[
h_{\mu}(r) \succeq r^{\alpha/\beta}, \qquad r>0.
\]
If \eqref{eq:fundamental-power-gap} is assumed only for $0<t\le t_{0}$, then
the corresponding local conclusion holds:
\[
h_{\mu}(r) \succeq r^{\alpha/\beta}, \qquad0<r\le1.
\]

\end{proposition}

\begin{proof}
Let $u_{j}$ be the cutoff function defined in \eqref{eq:dyadic-cutoffs}.
Since
\[
\chi_{B_{j+1}} \le u_{j} \le\chi_{B_{j}}
\]
and
\[
g_{j} := \delta_{j}^{-\alpha}\chi_{B_{j}} \in D^{\alpha}(u_{j}),
\]
the lattice property of $Y$ and \eqref{eq:ri-target-embedding} give
\[
\begin{aligned}
\varphi_Y(a_{j+1}) &= \|\chi_{B_{j+1}}\|_Y
\\
&\le \|u_j\|_Y
\\
&\preceq \|g_j\|_X
\\
&= \delta_j^{-\alpha}\varphi_X(a_j).
\end{aligned}
\]
This proves \eqref{eq:ri-target-first-recurrence}.

Since $a_{j+1}\le a_{j}$ and $\varphi_{X}(t)/t$ is decreasing,
\[
\varphi_{X}(a_{j}) \le\frac{a_{j}}{a_{j+1}} \varphi_{X}(a_{j+1}).
\]
Substituting this estimate into \eqref{eq:ri-target-first-recurrence} gives
\[
\varphi_{Y}(a_{j+1}) \preceq\delta_{j}^{-\alpha} \frac{a_{j}}{a_{j+1}}
\varphi_{X}(a_{j+1}),
\]
which is equivalent to \eqref{eq:ri-target-second-recurrence}.

Assume now \eqref{eq:fundamental-power-gap}. Then
\[
a_{j} \succeq\delta_{j}^{\alpha}a_{j+1}^{1-\beta}.
\]
If $\beta=1$, then
\[
a_{j}\succeq\delta_{j}^{\alpha},
\]
and hence
\[
a_{1}\succeq r^{\alpha}=r^{\alpha/\beta}.
\]
If $0<\beta<1$, set
\[
\gamma:=1-\beta.
\]
Iterating as in the proof of Theorem~\ref{thm:normed-oscillation-lower-growth}%
, we obtain
\[
a_{1}\succeq r^{\alpha/(1-\gamma)} =r^{\alpha/\beta}.
\]

Since
\[
a_{1}\le\mu(B(x,r)),
\]
taking the infimum over $x\in\Omega$ proves the global conclusion.

Suppose finally that \eqref{eq:fundamental-power-gap} holds only for $0<t\le
t_{0}$. If $a_{1}\le t_{0}$, then
\[
a_{j}\le a_{1}\le t_{0}, \qquad j\ge1,
\]
so the preceding argument, starting with $j=1$, applies. If $a_{1}>t_{0}$,
then, for $0<r\le1$,
\[
\mu(B(x,r))\ge a_{1}>t_{0} \ge t_{0}r^{\alpha/\beta}.
\]
Taking the infimum over $x\in\Omega$ proves the local conclusion.
\end{proof}

\begin{remark}
\label{rem:nonhomogeneous-ri-target} The local conclusion remains valid if
\eqref{eq:ri-target-embedding} is replaced by
\[
\|f\|_{Y} \preceq\|f\|_{\dot M^{\alpha,X}} + \|f\|_{X}.
\]
Indeed, for the cutoffs $u_{j}$ and $0<r\le1$,
\[
\|u_{j}\|_{X} \le\varphi_{X}(a_{j}) \le\delta_{j}^{-\alpha}\varphi_{X}%
(a_{j}),
\]
so the lower-order term is absorbed by the homogeneous estimate.
\end{remark}

\begin{remark}
There is an analogous statement for the averaged Besov scale, but
its hypotheses are different. Let $0<\theta<1$, $0<p<\infty$, and
$0<q\le\infty$. Assume that $X^{(1/p)}$ is an r.i. space and that
\[
D_{X,p} := \sup_{s>0} \|A_{s}\|_{X^{(1/p)}\to X^{(1/p)}} < \infty.
\]
If
\[
\|f\|_{Y} \preceq\|f\|_{\dot{\mathcal{B}}_{X,q}^{\theta,p}}
\]
for every boundedly supported Lipschitz function $f$, then the cutoff estimate
proved in Theorem~\ref{thm:normed-besov-lower-growth} gives
\[
\varphi_{Y}(a_{j+1}) \preceq\delta_{j}^{-\theta}\varphi_{X}(a_{j}).
\]
Thus Proposition~\ref{prop:embedding-lower-growth} remains valid with $\alpha$
replaced by $\theta$. In particular,
\[
\varphi_{Y}(t) \succeq t^{-\beta}\varphi_{X}(t)
\]
implies
\[
h_{\mu}(r) \succeq r^{\theta/\beta}.
\]

For $X=L^{p}$, the required operator assumption follows from
\[
\sup_{s>0} \|A_{s}\|_{L^{1}\to L^{1}} < \infty.
\]
This is the only additional ingredient in the averaged Besov argument; no
doubling, Poincar\'e, Ahlfors regularity, or uniform perfectness assumption is
required. Moreover, since the proof uses only the lattice estimate
\[
\|\chi_{B_{j+1}}\|_{Y} \le\|u_{j}\|_{Y},
\]
the same conclusion remains valid for rearrangement-invariant quasi-Banach
targets, including $L^{p_{\theta}^{*},q}$ when $0<q<1$.
\end{remark}

\begin{remark}
Let \(X=L^p\). If
\[
\alpha p<Q,
\]
then, for
\[
Y=L^{p_\alpha^*,p} \qquad\text{or}\qquad Y=L^{p_\alpha^*}, \qquad
\frac1{p_\alpha^*} = \frac1p-\frac{\alpha}{Q},
\]
we have
\[
\frac{\varphi_Y(t)}{\varphi_{L^p}(t)} \simeq t^{-\alpha/Q}.
\]
Proposition~\ref{prop:embedding-lower-growth} therefore shows that
either embedding forces
\[
h_\mu(r)\succeq r^Q.
\]
The same conclusion holds for the averaged Besov embedding into
\(L^{p_\theta^*,q}\), under the boundedness assumption on the
averaging operators. Thus, in this range, the
rearrangement-invariant target retains the full lower power. For
\(\alpha=1\), this recovers the conclusion of
\cite[Theorem~22(a)]{AGH}; the fractional case may be compared with
the same result after replacing \(d\) by \(d^\alpha\).

If
\[
\alpha p=Q,
\]
the fundamental function of the global critical target satisfies
\[
\varphi_{Y_{\mathrm{crit}}}(t) \simeq
\left(\log\frac{e}{t}\right)^{-1/p'}, \qquad 0<t\leq\frac12.
\]
The cutoff argument then yields only weaker local power estimates
and does not recover the exponent \(Q\) directly. This explains the
use of a localized Trudinger inequality in \cite[Theorem~25]{AGH}.
The case \(p=1\) is not included here, since the corresponding
endpoint space is not a rearrangement-invariant Banach function
space.

Finally, if
\[
\alpha p>Q,
\]
the global rearrangement-invariant target is \(L^\infty\), and
Proposition~\ref{prop:embedding-lower-growth} gives only
\[
h_\mu(r)\succeq r^{\alpha p}, \qquad 0<r\leq1,
\]
which is weaker than
\[
h_\mu(r)\succeq r^Q.
\]
Recovering the exact power requires the metric information contained
in the H\"older--Morrey estimate, whose target is not
rearrangement-invariant; compare \cite[Theorem~29]{AGH}.

Thus, the global rearrangement-invariant target retains the complete
lower-growth information when
\[
\alpha p<Q,
\]
whereas at and beyond the endpoint additional local or metric
information is needed. This also clarifies the formulations used in
\cite{GorkaPA,AGH,Karak,GorkaNA,GSStudia}.
\end{remark}

\subsection{Slobodeckij-type consequences}

We finally indicate how the preceding results apply to Slobodeckij-type
spaces. The only point that needs to be verified is that the corresponding
seminorms control the small-scale part of the averaged Besov seminorm.

Let $0<\theta<1$ and $0<p<\infty$. We consider first the intrinsic Slobodeckij
seminorm
\[
[f]_{W_{\mu}^{\theta,p}}^{p} := \iint_{0<d(x,y)<1} \frac{|f(x)-f(y)|^{p}}
{d(x,y)^{\theta p}\mu(B(x,d(x,y)))} \,d\mu(y)\,d\mu(x).
\]
We also introduce the dimension-dependent seminorm
\[
[f]_{W_{Q}^{\theta,p}}^{p} := \iint_{0<d(x,y)<1} \frac{|f(x)-f(y)|^{p}}
{d(x,y)^{Q+\theta p}} \,d\mu(y)\,d\mu(x).
\]

\begin{proposition}
\label{prop:slobodeckij-to-besov} For every measurable function $f$,
\[
\left(  \int_{0}^{1} \left[  \frac{E_{p}(f,r)}{r^{\theta}} \right]  ^{p}
\frac{dr}{r} \right)  ^{1/p} \preceq[f]_{W_{\mu}^{\theta,p}}.
\]
If, in addition,
\[
h_{\mu}(r)\succeq r^{Q}, \qquad0<r<1,
\]
then
\[
[f]_{W_{\mu}^{\theta,p}} \preceq[f]_{W_{Q}^{\theta,p}}.
\]
Consequently,
\[
\left(  \int_{0}^{1} \left[  \frac{E_{p}(f,r)}{r^{\theta}} \right]  ^{p}
\frac{dr}{r} \right)  ^{1/p} \preceq[f]_{W_{\mu}^{\theta,p}} \preceq
[f]_{W_{Q}^{\theta,p}}.
\]

\end{proposition}

\begin{proof}
By Tonelli's theorem,
\[
\begin{aligned}
\int_0^1 \left[ \frac{E_p(f,r)}{r^\theta} \right]^p \frac{dr}{r} &=
\iint_{0<d(x,y)<1} |f(x)-f(y)|^p
\\
&\qquad\times \left( \int_{d(x,y)}^1 \frac{dr} {r^{\theta
p+1}\mu(B(x,r))} \right) d\mu(y)\,d\mu(x).
\end{aligned}
\]
For $r\ge d(x,y)$,
\[
B(x,d(x,y))\subset B(x,r),
\]
and hence
\[
\frac1{\mu(B(x,r))} \le\frac1{\mu(B(x,d(x,y)))}.
\]
Therefore
\[
\begin{aligned}
\int_0^1 \left[ \frac{E_p(f,r)}{r^\theta} \right]^p \frac{dr}{r}
&\preceq \iint_{0<d(x,y)<1} \frac{|f(x)-f(y)|^p} {\mu(B(x,d(x,y)))}
\\
&\qquad\times \left( \int_{d(x,y)}^1 \frac{dr}{r^{\theta p+1}}
\right) d\mu(y)\,d\mu(x)
\\
&\preceq [f]_{W_\mu^{\theta,p}}^p.
\end{aligned}
\]

Assume now that
\[
h_{\mu}(r)\succeq r^{Q}, \qquad0<r<1.
\]
Then
\[
\mu(B(x,d(x,y))) \succeq d(x,y)^{Q}
\]
whenever $0<d(x,y)<1$. Consequently,
\[
\frac1{ d(x,y)^{\theta p}\mu(B(x,d(x,y))) } \preceq\frac1{d(x,y)^{Q+\theta p}%
},
\]
which gives
\[
[f]_{W_{\mu}^{\theta,p}} \preceq[f]_{W_{Q}^{\theta,p}}.
\]

\end{proof}

The preceding proposition allows us to transfer the averaged Besov embeddings
to the Slobodeckij scales.

\begin{corollary}
\label{cor:slobodeckij-subcritical} Assume that
\[
h_{\mu}(r)\succeq r^{Q}, \qquad0<r<1,
\]
and let
\[
1\le p<\frac{Q}{\theta}, \qquad p_{\theta}^{*}:=\frac{Qp}{Q-\theta p}.
\]
Then, for every $f\in L^{p}(\Omega)$ such that $[f]_{W_{\mu}^{\theta,p}%
}<\infty$,
\[
\|f\|_{L^{p_{\theta}^{*},p}} \preceq[f]_{W_{\mu}^{\theta,p}}+\|f\|_{L^{p}}.
\]
In particular, for every $f\in L^{p}(\Omega)$ such that $[f]_{W_{Q}^{\theta
,p}}<\infty$,
\[
\|f\|_{L^{p_{\theta}^{*},p}} \preceq[f]_{W_{Q}^{\theta,p}}+\|f\|_{L^{p}}.
\]

\end{corollary}

\begin{proof}
The local pointwise oscillation estimate and
Proposition~\ref{prop:slobodeckij-to-besov} imply that, for some $t_{0}>0$,
\[
\int_{0}^{t_{0}} \left[  t^{1/p_{\theta}^{*}}O(f,t)\right]  ^{p}\frac{dt}{t}
\preceq[f]_{W_{\mu}^{\theta,p}}^{p}.
\]
For $t\ge t_{0}$,
\[
O(f,t)\le f^{**}(t)\le t^{-1/p}\|f\|_{L^{p}},
\]
and hence
\[
\int_{t_{0}}^{\infty} \left[  t^{1/p_{\theta}^{*}}O(f,t)\right]  ^{p}\frac
{dt}{t} \preceq\|f\|_{L^{p}}^{p}.
\]
Since $f\in L^{p}(\Omega)$, we have $f^{**}(\infty)=0$, and the oscillation
characterization of Lorentz spaces yields
\[
\|f\|_{L^{p_{\theta}^{*},p}} \preceq[f]_{W_{\mu}^{\theta,p}}+\|f\|_{L^{p}}.
\]
The second estimate follows from
\[
[f]_{W_{\mu}^{\theta,p}} \preceq[f]_{W_{Q}^{\theta,p}}.
\]

\end{proof}

\section{Comparison between Haj\l asz and averaged Besov scales}

We conclude with some additional relations between the Haj\l asz and averaged
Besov scales. These results are not needed in the preceding development, but
they clarify the interaction between the two smoothness structures and provide
non-doubling examples in which the averaging operators remain uniformly bounded.

We first relate the Haj\l asz scale to the averaged Besov scale introduced
above. The point is that, under a mild boundedness assumption on the averaging
operators, Haj\l asz control of order $\phi$ implies averaged Besov control of
any order below the lower index of $\phi$.

More precisely, we shall prove embeddings of the form
\[
M^{\phi,X}(\Omega) \hookrightarrow\mathcal{B}_{X,q}^{\theta,p}(\Omega).
\]
This should be compared with the classical relationship between Besov and
Haj\l asz--Sobolev spaces on doubling metric measure spaces, where one has,
under suitable assumptions,
\[
B_{p,p}^{\alpha}(\Omega)\hookrightarrow M^{\alpha,p}(\Omega),
\]
and, conversely,
\[
M^{\alpha,p}(\Omega)\hookrightarrow B_{p,p}^{\alpha-\varepsilon}(\Omega),
\qquad0<\varepsilon<\alpha;
\]
see, for instance, \cite{GKS}. Here we prove an abstract sufficient embedding
of the second type in the rearrangement-invariant setting. The role of the
doubling assumption is replaced by the boundedness of the averaging operators,
and the power scale $r^{\alpha}$ is replaced by a general function $\phi
\in\mathcal{A}_{0}$.

If $p>1$, we say that $X$ is $p$-convex when the space $X^{(1/p)}$ is a
rearrangement-invariant Banach function space, up to an equivalent
norm.\footnote{This is the standard lattice notion of $p$-convexity; see, for
instance, \cite{LT2}.} Under this assumption, since $\|A_{r}\|_{L^{\infty}\to
L^{\infty}}\le1$, the $L^{1}$-boundedness of $A_{r}$ implies, by interpolation
(see \cite{BS}), that
\[
\|A_{r} h\|_{X^{(1/p)}} \preceq\left(  1+\|A_{r}\|_{L^{1}\to L^{1}}\right)
\|h\|_{X^{(1/p)}} .
\]
Consequently,
\begin{equation}
\label{eq:average-convexification}\left\|  (A_{r} |g|^{p})^{1/p}\right\|  _{X}
= \|A_{r}(|g|^{p})\|_{X^{(1/p)}}^{1/p} \preceq\left(  1+\|A_{r}\|_{L^{1}\to
L^{1}}\right)  ^{1/p} \|g\|_{X} .
\end{equation}

For $0<p\le1$, no additional convexity assumption is needed when $X$ is a
rearrangement-invariant Banach function space, since $1/p\ge1$ and the
convexification $X^{(1/p)}$ is again a rearrangement-invariant Banach function
space. The Lebesgue model $X=L^{p}$, $0<p<1$, is not covered by this statement
because it is quasi-Banach. Nevertheless, the required estimate follows
directly in this case from $(L^{p})^{(1/p)}=L^{1}.$

\begin{theorem}
\label{thm:hajlasz-to-averaged-besov-pointwise} Let $\phi\in\mathcal{A}_{0}$,
let $0<p<\infty$, and let $X$ be a r.i. space on $\Omega$. If $p>1$, assume in
addition that $X$ is $p$-convex. Assume that $A_{r}$ is bounded on $L^{1}%
(\mu)$ for every $r>0$. Then, for every $f\in M^{\phi,X}(\Omega)$ and every
$r>0$,
\[
E_{X}^{p}(f,r) \preceq\left(  1+\|A_{r}\|_{L^{1}\to L^{1}}^{1/p}\right)
\min\{1,\phi(r)\}\, \|f\|_{M^{\phi,X}} .
\]

\end{theorem}

\begin{proof}
Fix $f\in M^{\phi,X}(\Omega)$ and $r>0$. By definition,
\[
E_{X}^{p}(f,r)=\left\Vert \left(
\mathchoice{{\setbox0=\hbox{$\displaystyle{\textstyle -}{\int}$}\vcenter{\hbox{$\textstyle -$}}\kern-.5\wd0}}{{\setbox0=\hbox{$\textstyle{\scriptstyle -}{\int}$}\vcenter{\hbox{$\scriptstyle -$}}\kern-.5\wd0}}{{\setbox0=\hbox{$\scriptstyle{\scriptscriptstyle -}{\int}$}\vcenter{\hbox{$\scriptscriptstyle -$}}\kern-.5\wd0}}{{\setbox0=\hbox{$\scriptscriptstyle{\scriptscriptstyle -}{\int}$}\vcenter{\hbox{$\scriptscriptstyle -$}}\kern-.5\wd0}}\!\int%
_{B(x,r)}|f(x)-f(y)|^{p}\,d\mu(y)\right)  ^{1/p}\right\Vert _{X}.
\]
We use the elementary inequality
\[
|a-b|^{p}\leq C_{p}\left(  |a|^{p}+|b|^{p}\right)  ,\qquad a,b\in\mathbb{R},
\]
where $C_{p}=\max\{1,2^{p-1}\}$. Hence
\[
E_{X}^{p}(f,r)\leq C_{p}^{1/p}\left(  \Vert f\Vert_{X}+\Vert(A_{r}%
|f|^{p})^{1/p}\Vert_{X}\right)  .
\]
By (\ref{eq:average-convexification}), we get
\[
\Vert A_{r}(|f|^{p})\Vert_{X^{(1/p)}}^{1/p}=\Vert(A_{r}|f|^{p})^{1/p}\Vert
_{X}\preceq\left(  1+\Vert A_{r}\Vert_{L^{1}\rightarrow L^{1}}^{1/p}\right)
\Vert f\Vert_{X}.
\]
Consequently,
\begin{equation}
E_{X}^{p}(f,r)\preceq\left(  1+\Vert A_{r}\Vert_{L^{1}\rightarrow L^{1}}%
^{1/p}\right)  \Vert f\Vert_{X}. \label{eq:averaged-modulus-from-X}%
\end{equation}

Now let $g\in D^{\phi}(f)$. Then, for $\mu$-a.e. $x,y\in\Omega$,
\[
|f(x)-f(y)|\leq\phi(d(x,y))(g(x)+g(y)).
\]
If $y\in B(x,r)$, then $d(x,y)\leq r$, and therefore, since $\phi$ is
increasing,
\[
\phi(d(x,y))\leq\phi(r),
\]
thus%
\[
|f(x)-f(y)|^{p}\leq\phi(r)^{p}(g(x)+g(y))^{p}.
\]
Applying the preceding argument, we obtain
\begin{align*}
E_{X}^{p}(f,r)  &  \leq C_{p}^{1/p}\phi(r)\left(  \Vert g\Vert_{X}+\Vert
(A_{r}|g|^{p})^{1/p}\Vert_{X}\right) \\
&  \preceq C_{p}^{1/p}\phi(r)\left(  \Vert g\Vert_{X}+\left(  1+\Vert
A_{r}\Vert_{L^{1}\rightarrow L^{1}}^{1/p}\right)  \Vert g\Vert_{X}\right) \\
&  \preceq\phi(r)\left(  1+\Vert A_{r}\Vert_{L^{1}\rightarrow L^{1}}%
^{1/p}\right)  \Vert g\Vert_{X}.
\end{align*}
Taking the infimum over all $g\in D^{\phi}(f)$, we obtain
\begin{equation}
E_{X}^{p}(f,r)\preceq\phi(r)\left(  1+\Vert A_{r}\Vert_{L^{1}\rightarrow
L^{1}}^{1/p}\right)  \Vert f\Vert_{\dot{M}^{\phi,X}}.
\label{eq:averaged-modulus-from-hajlasz}%
\end{equation}

Combining (\ref{eq:averaged-modulus-from-X}) and
(\ref{eq:averaged-modulus-from-hajlasz}), we conclude that
\[
E_{X}^{p}(f,r)\preceq\left(  1+\Vert A_{r}\Vert_{L^{1}\rightarrow L^{1}}%
^{1/p}\right)  \min\{1,\phi(r)\}\left(  \Vert f\Vert_{X}+\Vert f\Vert_{\dot
{M}^{\phi,X}}\right)  .
\]
This proves the theorem.
\end{proof}

\begin{corollary}
\label{cor:Besov_embedding} Let $\phi\in\mathcal{A}_{0}$, let
$0<p<\infty$, $0<q\le\infty$, $0<\theta<\infty$, and let $X$ be an
r.i. space on $\Omega$. If $p>1$, assume in addition that $X$ is
$p$-convex. Assume that
\[
D:=\sup_{r>0}\|A_{r}\|_{L^{1}\to L^{1}}<\infty
\]
and that
\[
\theta<\underline{\beta}_{\phi}.
\]
Then
\[
M^{\phi,X}(\Omega) \hookrightarrow\mathcal{B}_{X,q}^{\theta,p}(\Omega).
\]
More precisely, there exists a constant $C>0$ such that, for every $f\in
M^{\phi,X}(\Omega)$,
\[
\|f\|_{\mathcal{B}_{X,q}^{\theta,p}} \le C\|f\|_{M^{\phi,X}} .
\]

\end{corollary}

\begin{proof}
Let $f\in M^{\phi,X}(\Omega)$. By
Theorem~\ref{thm:hajlasz-to-averaged-besov-pointwise} and the definition of
$D$, for every $r>0$,
\[
E_{X}^{p}(f,r) \preceq(1+D^{1/p})\min\{1,\phi(r)\} \|f\|_{M^{\phi,X}} .
\]
Therefore, if $0<q<\infty$,
\[
\|f\|_{\dot{\mathcal{B}}_{X,q}^{\theta,p}} \preceq\|f\|_{M^{\phi,X}} \left(
\int_{0}^{\infty}\left(  r^{-\theta}\min\{1,\phi(r)\}\right)  ^{q} \frac
{dr}{r} \right)  ^{1/q}.
\]
For $q=\infty$, the same argument gives
\[
\|f\|_{\dot{\mathcal{B}}_{X,\infty}^{\theta,p}} \preceq\|f\|_{M^{\phi,X}}
\sup_{r>0} r^{-\theta}\min\{1,\phi(r)\}.
\]

It remains to check that
\[
r^{-\theta}\min\{1,\phi(r)\}
\]
belongs to $L^{q}((0,\infty),dr/r)$, with the usual modification when
$q=\infty$.

For $r\ge1$, since $\min\{1,\phi(r)\}\le1$, we have
\[
r^{-\theta}\min\{1,\phi(r)\}\le r^{-\theta}.
\]
This gives integrability at infinity because $\theta>0$.

For $0<r<1$, choose $\varepsilon>0$ such that
\[
\theta+\varepsilon<\underline{\beta}_{\phi}.
\]
By the definition of the lower fundamental index (\ref{indi}), there exists a
constant $C_{\varepsilon}>0$ such that
\[
\phi(r)\le C_{\varepsilon}r^{\theta+\varepsilon}, \qquad0<r<1.
\]
Hence
\[
r^{-\theta}\min\{1,\phi(r)\} \le r^{-\theta}\phi(r) \le C_{\varepsilon
}r^{\varepsilon}, \qquad0<r<1.
\]
This gives integrability near zero, and also boundedness near zero in the case
$q=\infty$.

Consequently,
\[
\|f\|_{\dot{\mathcal{B}}_{X,q}^{\theta,p}} \preceq\|f\|_{M^{\phi,X}}.
\]
Adding the $X$-term in the non-homogeneous Besov norm gives
\[
\|f\|_{\mathcal{B}_{X,q}^{\theta,p}} \preceq\|f\|_{M^{\phi,X}},
\]
as claimed.
\end{proof}

\begin{remark}
In the model case $\phi(r)=r^{\alpha}$, $X=L^{p}$, and $q=p$, the condition
$\theta<\underline{\beta}_{\phi}$ becomes simply $\theta<\alpha$. Hence,
taking $\theta=\alpha-\varepsilon$, with $0<\varepsilon<\alpha$,
Corollary~\ref{cor:Besov_embedding} gives
\[
M^{\alpha,L^{p}}(\Omega) \hookrightarrow\mathcal{B}_{L^{p},p}^{\alpha
-\varepsilon,p}(\Omega).
\]
Thus, in the classical notation, this recovers the embedding
\[
M^{\alpha,p}(\Omega) \hookrightarrow B_{p,p}^{\alpha-\varepsilon}(\Omega),
\qquad0<\varepsilon<\alpha,
\]
under the boundedness assumption on the averaging operators.
\end{remark}

The converse need not hold in general. Thus, in the preceding results, the
doubling assumption is replaced by the weaker boundedness condition
\[
\sup_{r>0}\|A_{r}\|_{L^{1}\to L^{1}}<\infty.
\]
The next proposition shows that this condition may hold beyond the doubling setting.

\begin{proposition}
\label{prop:local-comparability-averages} Assume that the measure $\mu$
satisfies the following uniform local comparability condition: there exists
$C\ge1$ such that, for all $x,y\in\Omega$ and all $r>0$ with $d(x,y)\le r$,
\begin{equation}
\label{eq:local_comp}\mu(B(x,r))\le C\,\mu(B(y,r)).
\end{equation}
Then
\[
\sup_{r>0}\|A_{r}\|_{L^{1}\to L^{1}}<\infty.
\]
In fact,
\[
\|A_{r}\|_{L^{1}\to L^{1}}\le C,\qquad r>0.
\]

\end{proposition}

\begin{proof}
Let $g\in L^{1}(\Omega)$. By Fubini's theorem,
\begin{align*}
\|A_{r} g\|_{L^{1}}  &  \le\int_{\Omega} \frac{1}{\mu(B(x,r))} \int_{B(x,r)}
|g(y)|\,d\mu(y)\,d\mu(x)\\
&  = \int_{\Omega} |g(y)| \left(  \int_{B(y,r)} \frac{d\mu(x)}{\mu(B(x,r))}
\right)  d\mu(y).
\end{align*}
If $x\in B(y,r)$, then $d(x,y)<r$, and therefore
\[
\frac{\mu(B(y,r))}{\mu(B(x,r))} \le\sup_{d(a,b)\le r} \frac{\mu(B(b,r))}%
{\mu(B(a,r))}.
\]
Hence
\[
\int_{B(y,r)} \frac{d\mu(x)}{\mu(B(x,r))} \le\sup_{d(a,b)\le r} \frac
{\mu(B(b,r))}{\mu(B(a,r))}.
\]
Consequently,
\[
\|A_{r}\|_{L^{1}\to L^{1}} \le\sup_{d(a,b)\le r} \frac{\mu(B(b,r))}%
{\mu(B(a,r))}.
\]
The uniform local comparability condition gives the desired bound.
\end{proof}

\begin{remark}
If $\mu$ is doubling, then it satisfies the uniform local comparability
condition. Indeed, if $d(a,b)\le r$, then
\[
B(b,r)\subset B(a,2r),
\]
and hence
\[
\frac{\mu(B(b,r))}{\mu(B(a,r))} \le\frac{\mu(B(a,2r))}{\mu(B(a,r))} \le
C_{\mu}.
\]
Therefore every doubling measure satisfies
\[
\sup_{r>0}\|A_{r}\|_{L^{1}\to L^{1}}<\infty.
\]
The converse need not hold in general. In particular, the boundedness of the
averaging operators may hold in non-doubling situations; see \cite{Aldaz}.
\end{remark}

\begin{example}
\label{ex:gaussian-product} The boundedness of the averaging operators does
not require the local comparability condition (\ref{eq:local_comp}). Aldaz
\cite{Aldaz} constructed examples with finite measures, such as $\mathbb{R}%
^{d}$ endowed with the standard Gaussian measure $\gamma_{d}$, where local
comparability fails but the averaging operators are uniformly bounded on
$L^{1}$.

The following simple variant applies when the underlying measure is infinite.
Let
\[
\Omega=\mathbb{R}^{d}\times\mathbb{R}, \qquad\mu=\gamma_{d}\otimes
\mathcal{L}^{1},
\]
where $\mathcal{L}^{1}$ denotes the Lebesgue measure on $\mathbb{R}$. Endow
$\Omega$ with the metric
\[
d((x,t),(y,s))=\max\{|x-y|,|t-s|\}.
\]
Then
\[
B((x,t),r) = B_{\mathbb{R}^{d}}(x,r)\times(t-r,t+r),
\]
and therefore
\[
\mu(B((x,t),r)) = \gamma_{d}(B_{\mathbb{R}^{d}}(x,r))\,2r.
\]
The uniform local comparability condition fails, since it already fails for
the Gaussian factor.

On the other hand, the averaging operators $A_{r}$ are uniformly bounded on
$L^{1}(\mu)$. Indeed, define
\[
(\widetilde{A}_{r} f)(y,t) = \frac{1}{2r}\int_{t-r}^{t+r}f(y,s)\,ds.
\]
Then
\[
A_{r} f(x,t) = A_{r}^{\gamma_{d}}\bigl(\widetilde{A}_{r} f(\cdot,t)\bigr)(x),
\]
where $A_{r}^{\gamma_{d}}$ denotes the averaging operator on $(\mathbb{R}%
^{d},\gamma_{d})$. Since $\widetilde{A}_{r}$ is contractive on $L^{1}%
(\mathcal{L}^{1})$, and the operators $A_{r}^{\gamma_{d}}$ are uniformly
bounded on $L^{1}(\gamma_{d})$ by \cite{Aldaz}, Fubini's theorem gives
\[
\|A_{r} f\|_{L^{1}(\mu)} \le C\|f\|_{L^{1}(\mu)}, \qquad r>0,
\]
with a constant independent of $r$.
\end{example}

%-------------------------------

\section*{Statements and Declarations}

\noindent\textbf{Competing interests.}
The author has no relevant financial or non-financial interests to disclose.

\end{document}